\documentclass[10pt]{amsart}
\usepackage{amssymb}

\newcounter{TmpEnumi}
\numberwithin{equation}{section}

\def\today{\number\day\space\ifcase\month\or   January\or February\or
   March\or April\or May\or June\or   July\or August\or September\or
   October\or November\or December\fi\   \number\year}

\renewcommand{\S}{\subset}
\newcommand{\SM}{\setminus}

\newcommand{\I}{\infty}

\theoremstyle{definition}
\newtheorem{thm}{Theorem}[section]
\newtheorem{lem}[thm]{Lemma}
\newtheorem{prp}[thm]{Proposition}
\newtheorem{dfn}[thm]{Definition}
\newtheorem{cor}[thm]{Corollary}

\newtheorem{rmk}[thm]{Remark}
\newtheorem{ntn}[thm]{Notation}

\newcommand{\beq}{\begin{equation}}
\newcommand{\eeq}{\end{equation}}
\newcommand{\beqr}{\begin{eqnarray*}}
\newcommand{\eeqr}{\end{eqnarray*}}
\newcommand{\bal}{\begin{align*}}
\newcommand{\eal}{\end{align*}}
\newcommand{\bei}{\begin{itemize}}
\newcommand{\eei}{\end{itemize}}
\newcommand{\limi}[1]{\lim_{{#1} \to \infty}}

\newcommand{\af}{\alpha}
\newcommand{\bt}{\beta}
\newcommand{\gm}{\gamma}
\newcommand{\dt}{\delta}
\newcommand{\ep}{\varepsilon}

\newcommand{\io}{\iota}

\newcommand{\ld}{\lambda}
\newcommand{\sm}{\sigma}

\newcommand{\ph}{\varphi}
\newcommand{\ps}{\psi}
\newcommand{\rh}{\rho}
\newcommand{\om}{\omega}
\newcommand{\ta}{\tau}

\newcommand{\Dt}{\Delta}

\newcommand{\Z}{{\mathbb{Z}}}
\newcommand{\R}{{\mathbb{R}}}
\newcommand{\C}{{\mathbb{C}}}
\newcommand{\N}{{\mathbb{Z}}_{> 0}}
\newcommand{\Nz}{{\mathbb{Z}}_{\geq 0}}

\newcommand{\id}{{\mathrm{id}}}

\newcommand{\Ad}{{\mathrm{Ad}}}
\newcommand{\dist}{{\mathrm{dist}}}
\newcommand{\sa}{{\mathrm{sa}}}

\newcommand{\supp}{{\mathrm{supp}}}
\newcommand{\rank}{{\mathrm{rank}}}

\newcommand{\Aut}{{\mathrm{Aut}}}
\newcommand{\Inn}{{\mathrm{Inn}}}
\newcommand{\ord}{{\mathrm{ord}}}
\newcommand{\Innb}{{\overline{\mathrm{Inn}}}}

\newcommand{\dirlim}{\displaystyle{\lim_{\longrightarrow}}}
\newcommand{\Mi}{M_{\infty}}

\newcommand{\andeqn}{\,\,\,\,\,\, {\mbox{and}} \,\,\,\,\,\,}

\newcommand{\ts}[1]{{\textstyle{#1}}}
\newcommand{\ds}[1]{{\displaystyle{#1}}}

\newcommand{\sssum}[2]{{\ts{ {\ds{\sum}}_{#1}^{#2} }}}

\newcommand{\ca}{C*-algebra}

\newcommand{\pj}{projection}

\newcommand{\hm}{homomorphism}
\newcommand{\wolog}{without loss of generality}
\newcommand{\Wolog}{Without loss of generality}

\newcommand{\ifo}{if and only if}

\newcommand{\mops}{mutually orthogonal \pj s}

\newcommand{\hme}{homeomorphism}
\newcommand{\mh}{minimal homeomorphism}

\newcommand{\cfn}{continuous function}
\newcommand{\hsa}{hereditary subalgebra}

\newcommand{\mvnt}{Murray-von Neumann equivalent}

\newcommand{\sfsuca}{stably finite simple unital \ca}
\newcommand{\trp}{tracial Rokhlin property}
\newcommand{\nptrpw}{weak
 tracial Rokhlin property with respect to}
\newcommand{\nptrp}{weak tracial Rokhlin property}
\newcommand{\idsfsuca}{infinite dimensional stably finite simple unital \ca}

\newcommand{\nzp}{nonzero projection}

\newcommand{\ov}{\overline}

\title[Tracial Rokhlin property]{The
      tracial Rokhlin property is generic}

\author{N.~Christopher Phillips}

\date{17~September 2012}

\address{Department of Mathematics, University of Oregon,
       Eugene OR 97403-1222, USA.}
\email[]{ncp@darkwing.uoregon.edu}

\subjclass{Primary 46L40;
 Secondary 46L55.}
\thanks{This material is based upon work supported by the
  US National Science Foundation under
  Grants DMS-0302401, DMS-0701076, and DMS-1101742.}

\begin{document}

\begin{abstract}
We prove several results of the following general form:
automorphisms of (or actions of ${\mathbb{Z}}^d$ on) certain kinds of
simple separable unital C*-algebras~$A$
which have a suitable version of the Rokhlin property are generic
among all automorphisms (or actions),
or in a suitable class of automorphisms.
That is,
the ones with the version of the Rokhlin property
contain a dense $G_{\delta}$-subset of the set of all such
automorphisms (or actions).

Specifically, we prove the following.
If $A$ is stable under tensoring with the Jiang-Su algebra~$Z,$
and has tracial rank zero,
then automorphisms with the tracial Rokhlin property are generic.
If $A$ has tracial rank zero,
or, more generally,
$A$~is tracially approximately divisible together with a
technical condition,
then automorphisms with the tracial Rokhlin property are generic
among the approximately inner automorphisms.
If $A$ is stable under tensoring
with the Cuntz algebra ${\mathcal{O}}_{\infty}$
or with a UHF algebra of infinite type,
then actions of ${\mathbb{Z}}^d$ on $A$ with the Rokhlin property
are generic among all actions of~${\mathbb{Z}}^d.$
We further give a related but more restricted result
for actions of finite groups.
\end{abstract}

\maketitle

\setcounter{section}{-1}

\section{Introduction}\label{Sec:Intro}

\indent
We prove the following five results,
each of which shows that,
under suitable circumstances,
``most'' automorphisms of an (often simple) \ca\  %
(or actions of $\Z^d$ or of a finite group)
have a suitable version
of the (tracial) Rokhlin property.
Specifically, for any separable \ca\  $A,$
give $\Aut (A)$ the topology of pointwise norm convergence.
This topology comes from a complete metric.
(See Lemma~\ref{L:AutTop} below.)
Using suitable products,
we also define complete metric topologies
on the set of actions of $\Z^d$ on $A$ for $d \in \N$
and on the set of actions of $G$ on $A$ for a finite group~$G.$
Our results are:
\begin{enumerate}
\item\label{Intro-1}
Let $A$ be a simple separable unital \ca\  with tracial
rank zero in the sense of~\cite{LnTTR},
let $Z$ be the Jiang-Su algebra~\cite{JS},
and assume $Z \otimes A \cong A.$
Then there is a dense $G_{\dt}$-set $G \S \Aut (A)$
such that every $\af \in G$ has the \trp\  of~\cite{OP1}.
(See Theorem~\ref{Main} below.)
\item\label{Intro-2}
Let $A$ be a simple separable unital \ca\  which is
tracially approximately divisible
satisfying a possibly redundant technical condition.
(In particular, this includes all
simple separable unital \ca s
with tracial rank zero.)
Then there is a dense $G_{\dt}$-set $G \S \Innb (A),$
the set of approximately inner automorphisms of $A,$
such that every $\af \in G$ has the \trp.
(See Theorem~\ref{InnMain} below.)
\item\label{Intro-3}
Let $A$ be a separable unital \ca.
Suppose that ${\mathcal{O}}_{\infty} \otimes A \cong A,$
or that there is a UHF algebra $D$ of infinite type
such that $D \otimes A \cong A.$
Then there exists a dense $G_{\dt}$-set in $G \S \Aut (A)$
such that every $\af \in G$ has the Rokhlin property.
(See Corollaries \ref{T-Main3} and~\ref{C-ZRPUHF} below.)
\item\label{Intro-4}
The result of~(\ref{Intro-3})
holds for actions of $\Z^d$ for any~$d \in \N,$
not just for actions of~$\Z.$
(See Theorems \ref{T-MainZd} and~\ref{T-RPUHF} below.)
\item\label{Intro-5}
Let $G$ be a finite group of cardinality~$r,$
and let $D$ be the $r^{\infty}$~UHF algebra.
The result of~(\ref{Intro-3})
holds for actions of $G$ on a separable unital \ca\  %
such that $D \otimes A \cong A.$
(See Theorem~\ref{T-RPFinGp} below.)
\end{enumerate}
Some of these results (such as~(\ref{Intro-1}),
the real rank zero case of~(\ref{Intro-2}),
and~(\ref{Intro-3})) are a few years old.
Others are much more recent,
and were added as we realized that more could be gotten
from similar methods.
The methods can also be used to prove further results of the same
general nature,
beyond those we consider in this paper.

These results are intended as a partial explanation of the following
observation.
Among classifiable simple stably finite \ca s
(see especially the direct limit classification of~\cite{EGL2}),
most Elliott invariants correspond to simple \ca s with real rank one
and many tracial states.
On the other hand,
\ca s which people actually construct for other purposes
often have real rank zero and very often have a unique tracial state.
The connection is that the crossed product of a simple unital \ca\  with
tracial rank zero by an action with the \trp\  necessarily
has real rank zero (this follows from Theorem~4.5 of~\cite{OP1}),
and often has tracial rank zero
(combine Theorems 2.9 and~3.4 of~\cite{LO}).
Moreover, if the original algebra has a unique tracial state,
the same is true for the crossed product (Theorem~2.18 of~\cite{OP2}).

Some results of this general nature are known for homeomorphisms of
compact metric spaces.
The idea that something like our results might be true came from
the proof in~\cite{FH},
where it is shown that certain manifolds admit
uniquely ergodic minimal diffeomorphisms.
The method of proof is to construct a certain
set of diffeomorphisms, and to show that the
uniquely ergodic minimal diffeomorphisms
form a dense $G_{\dt}$-set in this set.

If $X$ is the Cantor set,
then Corollary~4.5(1) of~\cite{BDK}
implies that the uniquely ergodic minimal homeomorphisms
contain a dense $G_{\dt}$-subset of the closure of the set of all
minimal homeomorphisms in the topology of uniform
convergence of the \hme s and their inverses,
which corresponds to the topology of pointwise convergence
of automorphisms of $C (X)$ and their inverses.
See the end of the introduction to~\cite{BDK} for notation,
see Definition~1.2 of~\cite{BDK} for the topology $\ta_w,$
and see the beginning of Section~3 of~\cite{BDK} for
the definition of the set ${\mathcal{O}}d.$
However, by Corollary~4.2 of~\cite{BDK}, the \mh s of $X$
are nowhere dense in the set of all \hme s of~$X.$

The proof of the result for $Z$-stable \ca s with tracial rank zero
depends on tensoring with an automorphism of~$Z$
which has a version of the \trp.
No automorphism of $Z$ can have the \trp\  %
as it is defined in~\cite{OP1},
because $Z$ has no nontrivial \pj s.
However, the tensor shift on $\bigotimes_{n \in \Z} Z$
has a version of the \trp\  %
in which one uses positive elements in place of \pj s in the definition.
For an automorphism
of a simple separable unital \ca\  with tracial rank zero,
this property implies the usual \trp.
The result on the shift may be of independent interest,
although one hopes that it has a property strong enough
to imply the Rokhlin property in the presence of tracial rank zero.

This paper is organized as follows.
In Section~\ref{Sec:WTRP},
we introduce the version of the \trp\  that we use for
the tensor shift on~$\bigotimes_{n \in \Z} Z,$
and prove that it implies the \trp\  %
in the presence of tracial rank zero.
In Section~\ref{Sec:TShift},
we prove that the shift does in fact have this property.
The results are stated in somewhat greater generality.
In particular, they apply to the tensor shift
on $\bigotimes_{n \in \Z} A$
when $A$ is simple, unital, infinite dimensional,
and has a unique tracial state.
However, something much more general ought to be true.
Section~\ref{Sec:Main} contains the proof that the \trp\  is generic
for $Z$-stable simple \ca s with tracial rank zero,
using the results of the first two sections.

In Section~\ref{Sec:Old},
we prove the result~(\ref{Intro-2}) above,
and in Section~\ref{Sec:PISCA}
we prove the results (\ref{Intro-3}), (\ref{Intro-4}),
and~(\ref{Intro-5}) above.
These sections are largely independent
of Sections~\ref{Sec:WTRP}, \ref{Sec:TShift}, and~\ref{Sec:Main},
depending only on the basic setup at the beginning
of Section~\ref{Sec:Main} or its analog.
In Section~\ref{Sec:Old},
the main class of interest is the simple separable unital \ca s
with tracial rank zero.
The ideas of the proof seem clearer in the context of what we
call tracially approximately divisible \ca s;
these are the ``tracial'' analog
of the approximately divisible \ca s of~\cite{BKR}.
We therefore develop the basic theory of these algebras.
Other results can be proved,
but we do not yet know other interesting examples,
so we do not go farther in this direction.
In Section~\ref{Sec:PISCA},
the methods are similar to,
but easier than,
those of Section~\ref{Sec:Main}.

We also mention that it will be proved in~\cite{KtPh}
that a generic automorphism
of a unital approximately divisible AF~algebra
has the version of the Rokhlin property
given in Definition~2.5 of~\cite{Iz0}.

This paper has some overlap with independent work
in Section~3 of~\cite{HWZ}.
Let $A$ be a separable unital \ca.
It is proved in Theorem~3.4 there that if
$A$ is $Z$-stable,
then actions of $\Z$ on~$A$
with Rokhlin dimension~$1$
(see Definition~2.3 of~\cite{HWZ})
are generic among all actions of~$\Z,$
and that if $A$ is stable under tensoring
with a UHF algebra of infinite type,
then automorphisms with the Rokhlin property
are generic.
The second result is our Corollary~\ref{C-ZRPUHF}
(part of (\ref{Intro-3}) above).
It is probably true that
if $A$ is simple, unital, and has tracial rank zero,
then every action of~$\Z$ with finite Rokhlin dimension as in~\cite{HWZ}
in fact has the tracial Rokhlin property,
so that, with work, the first result could be used
to give a different proof our Theorem~\ref{Main}
((\ref{Intro-1})~above).

The main part of the proof of Theorem~3.4 of~\cite{HWZ}
is not very different from our proofs of
Proposition~\ref{P:TRPDense}
and Lemma~\ref{L-2910PIApprox},
which are the main steps in the proofs
of Theorems~\ref{Main}, \ref{T-MainZd}, and~\ref{T-RPUHF}.
However,
our proofs are arranged differently,
and in particular apply to actions of any discrete group
for which one has a reasonable definition
of some form of the Rokhlin property
and an action with this property on a strongly selfabsorbing \ca,
while the proof in~\cite{HWZ}
seems to be limited to actions of~$\Z.$
For example,
one could write down a definition
of Rokhlin dimension for actions of~$\Z^d$
by suitably combining ideas from Definition~2.3 of~\cite{HWZ}
and Definition~\ref{D-2908-Zd-RP},
find an action of~$\Z^d$ on~$Z$
with finite Rokhlin dimension,
and use our method to show that such actions are
generic on separable unital $Z$-stable \ca s.

Let $A$ be a \ca.
For \pj s $p, q \in A,$
we write $p \precsim q$
to mean that $p$ is \mvnt\  %
to a sub\pj\  of~$q.$
If $A$ is unital,
then $T (A)$ denotes the tracial state space of~$A,$
and for a unitary $u \in A,$
we denote by $\Ad (u)$ the inner automorphism $a \mapsto u a u^*.$

We will use several times the $L^2$-norm (or seminorm)
associated with a tracial state $\ta$ of a \ca\  $A,$ given by
$\| a \|_{2, \ta} = \ta (a^* a)^{1/2}.$
See the discussion before Lemma V.2.20 of~\cite{Tk}
for more on this seminorm in the von Neumann algebra context.
All the properties we need are immediate from its identification
with the seminorm in which one completes $A$ to obtain the
Hilbert space $H_{\ta}$ for the
Gelfand-Naimark-Segal representation associated with $\ta,$
and from the relation $\ta (b a) = \ta (a b).$
In particular, we always have
$\| a b c \|_{2, \ta}
 \leq \| a \| \cdot \| b \|_{2, \ta} \cdot \| c \|.$
Note that
$| \ta (a) | \leq \| a \|_{2, \ta},$
since Cauchy-Schwarz gives
\[
| \ta (a) |^2 \leq \ta (a^* a) \ta (1_A^* 1_A) = \| a \|_{2, \ta}^2.
\]

The theorem here was inspired by a conversation with
George Elliott during a visit to the Fields Institute during
March 2004.
I am grateful to Thierry Giordano for calling my
attention to the paper~\cite{BDK},
and to Nate Brown for supplying the proof of Lemma~\ref{L:NoMinPj}.
I am also grateful to Takeshi Katsura
and Ilan Hirshberg for valuable discussions.

\section{A generalization of
   the tracial Rokhlin property}\label{Sec:WTRP}

\indent
We recall the definition of the \trp\  from
Definition~1.1 of~\cite{OP1}:

\begin{dfn}\label{TRPDfn}
Let $A$ be a \sfsuca\   and let $\af \in \Aut (A).$
We say that $\af$ has the {\emph{tracial Rokhlin property}}
if for every finite set $F \S A,$ every $\ep > 0,$
every $n \in \N,$
and every nonzero positive element $x \in A,$
there are \mops\  $e_0, e_1, \ldots, e_n \in A$ such that:
\begin{enumerate}
\item\label{2906TRPDfn-1}
$\| \af (e_j) - e_{j + 1} \| < \ep$ for $j = 0, 1, \ldots, n - 1.$
\item\label{2906TRPDfn-2}
$\| e_j a - a e_j \| < \ep$ for $j = 0, 1, \ldots, n$ and all $a \in F.$
\item\label{2906TRPDfn-3}
With $e = \sum_{j = 0}^{n} e_j,$ the \pj\  $1 - e$ is \mvnt\  to a
\pj\  in ${\overline{x A x}}.$
\end{enumerate}
\end{dfn}

We do not say anything about $\af (e_n).$

\begin{dfn}\label{D:NPjTRP}
Let $A$ be a separable unital \ca,
let $\af \in \Aut (A),$ and let $T \S T (A).$
We say that $\af$ has the
{\emph{weak tracial Rokhlin property with respect to $T$}}
if for every finite set $S \S A,$ every $\ep > 0,$
and every $n \in \N,$
there are $c_0, c_1, \ldots, c_n \in A$ such that:
\begin{enumerate}
\item\label{DR:Pos}
$0 \leq c_j \leq 1$ for $j = 0, 1, \ldots, n.$
\item\label{DR:Orth}
$c_j c_k = 0$ for $j, k = 0, 1, \ldots, n$ with $j \neq k.$
\item\label{DR:Shift}
$\| \af (c_j) - c_{j + 1} \| < \ep$ for $j = 0, 1, \ldots, n - 1.$
\item\label{DR:Comm}
$\| c_j a - a c_j \| < \ep$ for $j = 0, 1, \ldots, n$ and all $a \in S.$
\item\label{DR:Rem}
With $c = \sum_{j = 0}^{n} c_j,$ we have $\ta (1 - c) < \ep$
for all $\ta \in T.$
\end{enumerate}
\end{dfn}

The main differences are that the $c_j$ need not be \pj s,
and that the size of the remainder is explicitly controlled by traces.

This definition is clearly not useful if there are not
enough tracial states on $A.$
Even for a simple stably finite separable unital \ca,
the correct property should require, instead of
condition~(\ref{DR:Rem}),
that for every nonzero positive element $x \in A,$
one can require that $1 - c$ be equivalent in a suitable sense,
perhaps in the Cuntz semigroup of~$A,$
to an element in the \hsa\  of $A$ generated by~$x.$

\begin{prp}\label{P:NPImpTRP}
Let $A$ be a simple separable unital \ca\  with tracial
rank zero in the sense of~\cite{LnTTR}
(tracially~AF in~\cite{LnTAF}).
Let $\af \in \Aut (A)$ have the \nptrpw\  $T (A).$
Then $\af$ has the \trp\  (Definition~\ref{D:NPjTRP}).
\end{prp}

\begin{proof}
We verify the condition of Theorem~2.14 of~\cite{OP2}.
Thus, let $S \S A$ be finite and let $\ep > 0.$
Choose $\ep_0 > 0$ so small that
$4 \sqrt{(2 n + 3) \ep_0} + 5 \ep_0 \leq \ep.$
Apply Definition~\ref{D:NPjTRP} with $S$ as given,
and with $\ep_0$ in place of $\ep,$
obtaining elements $c_0, c_1, \ldots, c_n.$
For $k = 0, 1, \ldots, n,$ use the fact that $A$ has real rank zero
(Theorem~3.4 of~\cite{LnTAF})
to choose a \pj\  $p_k \in {\overline{c_k A c_k}}$
such that
$\| p_k c_k - c_k \| < \ep_0$ and $\| c_k p_k - c_k \| < \ep_0.$
Then $\| p_k c_k p_k - c_k \| < 2 \ep_0.$
Arguing as in the proof of Lemma~1.8 of~\cite{PhtRp4}
(and noting that here we have $n + 1$ elements instead of $n$),
the $p_k$ are orthogonal \pj s such that
$\| p_k - c_k \|_{2, \ta} < \sqrt{(2 n + 3) \ep_0} + 2 \ep_0$
for $k = 0, 1, \ldots, n$ and $\ta \in T (A).$
Along the way, we also get the estimate
\[
1 - \sum_{j = 0}^n \ta (p_j)
 \leq 1 - \sum_{j = 0}^n \ta (p_j c_j p_j)
 < (2 n + 3) \ep_0.
\]
Therefore
\[
\left\| 1 - \sssum{j = 0}{n} \ta (p_j) \right\|_{2, \ta}
 < \sqrt{(2 n + 3) \ep_0}
\]
for $\ta \in T (A).$

Set $e_k = p_k$ for $k = 1, 2, \ldots n,$ and set
$e_0 = 1 - \sum_{j = 1}^n p_j.$
Then
\[
\sum_{j = 0}^n e_j = 1
\andeqn
\| e_0 - p_0 \|_{2, \ta} < \sqrt{(2 n + 3) \ep_0}.
\]
Thus, in any case, we have
$\| e_k - c_k \|_{2, \ta} < 2 \sqrt{(2 n + 3) \ep_0} + 2 \ep_0$
for $k = 0, 1, \ldots, n$ and $\ta \in T (A).$

Now we can verify
conditions~(1), (2), and~(3) of Theorem~2.14 of~\cite{OP2}.
Condition~(3) has already been done.
For the others,
let $\ta \in T (A).$
Then for $k = 0, 1, \ldots, n - 1$ we have
\begin{align*}
\| \af (e_k) - e_{k + 1} \|_{2, \ta}
 & \leq \| e_k - c_k \|_{2, \ta \circ \af}
       + \| e_{k + 1} - c_{k + 1} \|_{2, \ta}
       + \| \af (c_k) - c_{k + 1} \|          \\
 & < 2 \left(2 \sqrt{(2 n + 3) \ep_0} + 2 \ep_0 \right) + \ep_0
   \leq \ep,
\end{align*}
and
for $k = 0, 1, \ldots, n$ and $a \in S$ we have
\[
\| [e_k, a] \|_{2, \ta}
  \leq 2 \| e_k - c_k \|_{2, \ta} + \| [c_k, a] \|
  < 2 \left(2 \sqrt{(2 n + 3) \ep_0} + 2 \ep_0 \right) + \ep_0
  \leq \ep.
\]
An application of Theorem~2.14 of~\cite{OP2}
completes the proof.
\end{proof}

\section{Tensor shifts}\label{Sec:TShift}

\indent
The purpose of this section is to give conditions
under which the two sided shift
on an infinite tensor product of \ca s has the \nptrp\  %
with respect to an infinite tensor product of
copies of the same tracial state.
Better results may be possible;
one would hope for a result using arbitrary tracial states
on the infinite tensor product.
If better results are true, they need more work.
In the case we are most interested in,
that of the Jiang-Su algebra $Z,$
there is a unique tracial state.

\begin{dfn}\label{D:Shift}
Let $A$ be a unital \ca.
We define the {\emph{minimal shift}} on the infinite
minimal tensor product $B = \bigotimes_{n \in \Z} A$
as follows.
Set $B_n = A^{\otimes 2 n},$ the (minimal) tensor product
of $2 n$ copies of $A.$
(Take $B_0 = \C.$)
Define $\ph_n \colon B_n \to B_{n + 1}$
by $\ph_n (a) = 1_A \otimes a \otimes 1_A$ for $a \in B_n.$
Identify $B = \dirlim B_n$ via the maps $\ph_n.$
Then take $\sm \colon B \to B$ to be the direct limit of
the maps $\sm_n \colon B_n \to B_{n + 1}$ defined
by $\sm_n (a) = 1_A \otimes 1_A \otimes a$ for $a \in B_n.$

We define the {\emph{maximal shift}} on the infinite
maximal tensor product in the same manner.

When $A$ is nuclear, we simply refer to the {\emph{shift}}.
\end{dfn}

\begin{lem}\label{L:ShiftOnInterval}
Let $X = [0, 1]^{\Z},$
and let $h \colon X \to X$ be the shift,
given by $h (x)_k = x_{k - 1}$ for $x = (x_k)_{k \in \Z} \in X$
and $k \in \Z.$
Let $\mu_0$ be a Borel probability measure on $[0, 1],$
and let $\mu$ be the infinite product measure on $X.$
Suppose $\mu_0 (\{ t \}) = 0$ for all $t \in [0, 1].$
Then for every $\ep > 0$ and every $n \in \N,$
there are $N \in \N$ and a closed set $Y_0 \S [0, 1]^N$
such that the set
\[
Y = \prod_{k \leq 0} [0, 1]
        \times Y_0 \times \prod_{k \geq N + 1} [0, 1]
\]
has the properties:
\begin{enumerate}
\item\label{LI:Disj}
$Y, \, h (Y), \, \ldots, \, h^{n - 1} (Y)$ are disjoint.
\item\label{LI:Big}
$\mu \big( X \setminus
 \big[ Y \cup h (Y) \cup \cdots \cup h^{n - 1} (Y) \big] \big) < \ep.$
\end{enumerate}
\end{lem}

\begin{proof}
For each $k,$ let $\mu_0^k$ be the product measure on $[0, 1]^k.$

Choose $m \in \N$ such that $\frac{1}{m} < \frac{\ep}{2}.$
Set $N = n (m + 1) - 1.$
Define $E_0 \S [0, 1]^N$ by
\[
E_0 = \big\{ (x_1, x_2, \ldots, x_N) \in [0, 1]^N \colon
      \min ( \{ x_k \colon n | k \} )
          < \min ( \{ x_k \colon n \nmid k \} ) \big\}.
\]

We claim that $\mu_0^N (E_0) \geq m N^{-1}.$
To prove the claim, for $l = 1, 2, \ldots, N$ set
\[
F_l = \big\{ (x_1, x_2, \ldots, x_N) \in [0, 1]^N \colon
      x_l < \min ( \{ x_k \colon k \neq l \} ) \big\}.
\]
Further set
\[
\Dt = \big\{ (x_1, x_2, \ldots, x_N) \in [0, 1]^N \colon
      {\mbox{There are $k \neq l$ such that $x_k = x_l$}} \big\}.
\]
Fubini's Theorem and the hypothesis
$\mu_0 (\{ t \}) = 0$ for all $t \in [0, 1]$
imply that
\[
\mu_0^2 \big( \{ (x, x) \colon x \in [0, 1] \} \big) = 0.
\]
It easily follows that $\mu_0^N (\Dt) = 0.$
The \hme\  %
\[
\sm_N (x_1, x_2, \ldots, x_N) = (x_N, x_1, x_2, \ldots, x_{N - 1})
\]
preserves $\mu_0^N$ and transitively permutes the sets $F_l.$
Therefore they all have the same measure.
Since the $F_l$ are disjoint and
\[
[0, 1]^N \setminus \Dt \S \bigcup_{l = 1}^N F_l,
\]
it follows that $\mu_0^N (F_l) = N^{-1}$ for $l = 1, 2, \ldots, N.$
Clearly
\[
F_n \cup F_{2 n} \cup \cdots \cup F_{m n} \S E_0,
\]
so the claim follows.

Next, set
\[
E = \prod_{k \leq 0} [0, 1]
        \times E_0 \times \prod_{k \geq N + 1} [0, 1].
\]
We claim that $h^l (E) \cap E = \varnothing$
for $l = 1, 2, \ldots, n - 1.$
Let $x \in E.$
Using $1 \leq l < n$ and $N - l \geq m n$ at the second step,
$x \in E$ at the third step,
and again using $1 \leq l < n$ at the fourth step,
we get
\begin{align*}
\min \big( \{ h^l (x)_k \colon
   {\mbox{$1 \leq k \leq N$ and $n \nmid k$}} \} \big)
& = \min \big( \{ x_{k - l} \colon
   {\mbox{$1 \leq k \leq N$ and $n \nmid k$}} \} \big)
     \\
& \leq \min (x_n, \, x_{2 n}, \, \ldots, \, x_{m n})
     \\
& < \min \big( \{ x_k \colon
         {\mbox{$1 \leq k \leq N$ and $n \nmid k$}} \} \big)
     \\
& \leq \min ( x_{n - l}, \, x_{2 n - l}, \, \ldots, \, x_{m n - l} )
     \\
& = \min \big( \{ h^l (x)_k \colon
   {\mbox{$1 \leq k \leq N$ and $n | k$}} \} \big).
\end{align*}
This shows $h^l (x) \not\in E,$
and proves the claim.

It follows that
$E, \, h (E), \, \ldots, \, h^{n - 1} (E)$ are disjoint.

Now use inner regularity of $\mu$ to choose a compact
set $Y_0 \S [0, 1]^N$
such that $\mu_0^N (Y_0) > \mu (E_0) - \frac{1}{2} n^{-1} \ep.$
Define $Y$ as in the statement of the lemma.
Clearly $Y, \, h (Y), \, \ldots, \, h^{n - 1} (Y)$ are disjoint.
Also, using the first claim in the proof at the fourth step,
$N = n (m + 1) - 1$ at the fifth step, and the choice of $m$
at the sixth step,
we have
\begin{align*}
\mu \big( X \setminus
 \big[ Y \cup h (Y) \cup \cdots \cup h^{n - 1} (Y) \big] \big)
& = 1 - n \mu (Y)
  = 1 - n \mu_0^N (Y_0)
    \\
& < 1 - n \mu_0^N (E_0) + \frac{\ep}{2}
  \leq 1 - \frac{n m}{N} + \frac{\ep}{2}
    \\
& < \frac{1}{m + 1} + \frac{\ep}{2} < \ep.
\end{align*}
This completes the proof.
\end{proof}

\begin{prp}\label{P:TRPWithNoAtom}
Let $A$ be a unital \ca,
and let $\ta_0 \in T (A).$
Suppose that there is $a \in A$ with $0 \leq a \leq 1$
and such that the spectral measure $\mu_0$ on $[0, 1],$
defined by $\int_0^1 f \, d \mu_0 = \ta_0 (f (a))$
for $f \in C ([0, 1]),$
satisfies $\mu_0 (\{ t \}) = 0$ for all $t \in [0, 1].$
Let $\sm$ be either the minimal shift or the maximal shift
on $B = \bigotimes_{n \in \Z} A$
(minimal or maximal tensor product, as appropriate),
as in Definition~\ref{D:Shift}.
Then $\sm$ has the \nptrpw\  the infinite tensor product
tracial state $\ta$ on $B$ obtained from $\ta_0.$
\end{prp}

\begin{proof}
Let $S \S A$ be finite, let $\ep > 0,$ and let $n \in \N.$
Following the notation of Definition~\ref{D:Shift}
and using density of the algebraic direct limit,
\wolog\  there is $M \in \N$ such that $S \S B_M = A^{\otimes 2 M}.$
Identify $C \big( [0, 1]^{\Z} \big)$
with the infinite tensor product $\bigotimes_{n \in \Z} C ([0, 1])$
of Definition~\ref{D:Shift}
by identifying
\[
C ([0, 1])^{\otimes 2 n}
=
C \big( [0, 1]^{\{ - n + 1, \, - n + 2, \, \ldots, \, n - 1, \, n \} }
          \big)
\]
in the obvious way.
Let $h \colon [0, 1]^{\Z} \to [0, 1]^{\Z}$ be the shift,
as in Lemma~\ref{L:ShiftOnInterval},
and let
$\gm \colon C \big( [0, 1]^{\Z} \big) \to C \big( [0, 1]^{\Z} \big)$
be the shift as in Definition~\ref{D:Shift},
so that $\gm (f) = f \circ h^{-1}$ for $h \in C \big( [0, 1]^{\Z} \big).$
Let $\mu$ be the infinite product measure on $[0, 1]^{\Z}$
as in Definition~\ref{D:Shift},
and let $\om$ be the corresponding tracial state on
$C \big( [0, 1]^{\Z} \big).$
Define $\ph_0 \colon C ([0, 1]) \to A$ by functional calculus:
$\ph_0 (f) = f (a)$ for $f \in C ([0, 1]).$
Then there is an induced infinite tensor product \hm\  %
$\ph \colon C \big( [0, 1]^{\Z} \big) \to B,$
which satisfies $\ph \circ \gm = \sm \circ \ph$
and $\ta \circ \ph = \om.$

Apply Lemma~\ref{L:ShiftOnInterval} with $\mu_0,$ $\mu,$
and $\ep$ as given, and with $n + 1$ in place of $n,$
obtaining  $N \in \N,$ a closed set $Y_0 \S [0, 1]^N,$
and the closed set
\[
Y = \prod_{k \leq 0} [0, 1]
        \times Y_0 \times \prod_{k \geq N + 1} [0, 1]
  \S [0, 1]^{\Z}
\]
as there.

We construct an open set $U_0 \S [0, 1]^N$ containing $Y_0$
such that the open set
\[
U = \prod_{k \leq 0} [0, 1]
        \times U_0 \times \prod_{k \geq N + 1} [0, 1]
  \S [0, 1]^{\Z}
\]
has the property that the sets
$U, \, h (U), \, \ldots, \, h^{n} (U)$ are disjoint.
To this end,
write $Y_0 = \bigcap_{m = 0}^{\infty} {\overline{V_m^{(0)}}}$
for open subsets $V_0^{(0)} \supset V_1^{(0)} \supset \cdots$
of $[0, 1]^N$ which contain $Y_0.$
Set
\[
V_m = \prod_{k \leq 0} [0, 1]
        \times V_m^{(0)} \times \prod_{k \geq N + 1} [0, 1]
  \S [0, 1]^{\Z},
\]
giving
\[
{\overline{V_m}} = \prod_{k \leq 0} [0, 1]
        \times {\overline{V_m^{(0)}}} \times \prod_{k \geq N + 1} [0, 1]
  \S [0, 1]^{\Z}.
\]
Then, for $0 \leq j < k \leq n,$
\[
\bigcap_{m = 0}^{\infty} h^j \big( {\overline{V_m}} \big)
     \cap \bigcap_{m = 0}^{\infty} h^k \big( {\overline{V_m}} \big)
  = h^j (Y) \cap h^k (Y)
  = \varnothing,
\]
so by compactness there is $m_0 (j, k)$ such that
\[
h^j \big( {\overline{V_{m_0 (j, k)}}} \big)
   \cap h^k \big( {\overline{V_{m_0 (j, k)}}} \big) = \varnothing.
\]
Set $m_0 = \max_{0 \leq j < k \leq n} m_0 (j, k),$
and take $U_0 = V_{m_0}^{(0)}.$
This completes the construction.

Choose a \cfn\  $f_0 \colon [0, 1]^N \to [0, 1]$
such that $\supp (f_0) \S U_0$ and $f_0 (x) = 1$ for all $x \in Y_0.$
Define $f \in C \big( [0, 1]^{\Z} \big)$ by
\[
f (\ldots, \, x_{-1}, \, x_0, \, x_1, \, \ldots, \, x_N, \, x_{N + 1},
        \, \ldots)
  = f_0 ( x_1, x_2, \ldots, x_N).
\]
For $j = 0, 1, \ldots, n,$
define $b_j \in C \big( [0, 1]^{\Z} \big)$ by
$b_j = \gm^{M + j} (f).$
Then set $c_j = \ph (b_j).$

We verify the conditions of Definition~\ref{D:NPjTRP}.
Condition~(\ref{DR:Pos}) is immediate.
For condition~(\ref{DR:Orth}),
observe that $b_j = f \circ h^{-M - j}$ satisfies
$\supp (b_j) \S h^{M + j} (U),$
and any $n + 1$ consecutive iterates of $U$ under $h$ are disjoint.
For condition~(\ref{DR:Shift}),
we actually have
$\| \af (c_j) - c_{j + 1} \| = 0$ for $j = 0, 1, \ldots, n.$

For condition~(\ref{DR:Comm}),
set
\[
C_0 = \bigotimes_{k = - M - N - n + 1}^M A
\andeqn
C_1 = \bigotimes_{k = M + 1}^{M + N + n} A.
\]
Thus,
using the same indexing conventions as in the identification of
$C \big( [0, 1]^{\Z} \big)$ with $\bigotimes_{n \in \Z} C ([0, 1])$
near the beginning of the proof,
$A^{\otimes 2 (M + N + n)} = C_0 \otimes C_1.$
Then
\[
S \S C_0 \otimes 1_{C_1} \S A^{\otimes 2 (M + N + n)} \S B
\]
and
\[
c_0, c_1, \ldots, c_n
    \in 1_{C_0} \otimes C_1 \S A^{\otimes 2 (M + N + n)} \S B,
\]
so in fact $c_j a = a c_j$ for $j = 0, 1, \ldots, n$ and all $a \in S.$

Finally,
for condition~(\ref{DR:Rem}),
we observe that, with $c = \sum_{j = 0}^{n} c_j,$ we have
\begin{align*}
\ta (c)
& = \sum_{j = 0}^{n} \ta (\ph (b_j))
  = \sum_{j = 0}^{n} \om (b_j)
      \\
& \geq \mu \big( Y \cup h (Y) \cup \cdots \cup h^{n - 1} (Y) \big)
  > 1 - \ep.
\end{align*}
So $\ta (1 - c) < \ep.$
\end{proof}

\begin{lem}\label{L:FCalcTNorm}
Let $f \in C ([0, 1]).$
Then for every $\ep > 0$ there is $\dt > 0$ such that
whenever $A$ is a unital \ca, $\ta \in T (A),$
and $a, b \in A$ satisfy $0 \leq a, b \leq 1$
and $\| a - b \|_{2, \ta} < \dt,$
then $\| f (a) - f (b) \|_{2, \ta} < \ep.$
\end{lem}

\begin{proof}
We claim it suffices to prove this when $f$ is a polynomial.
Indeed, given $f$ arbitrary and $\ep > 0,$
choose a polynomial $g \in C ([0, 1])$
such that $\| f - g \| < \tfrac{1}{3} \ep,$
get $\dt > 0$ by applying the result for $g$
with $\tfrac{1}{3} \ep$ in place of $\ep,$
and observe that
\[
\| f (a) - f (b) \|_{2, \ta}
  \leq \| f (a) - g (a) \| + \| g (a) - g (b) \|_{2, \ta}
         + \| g (b) - f (b) \|.
\]

It is now clearly enough to prove the result for the case $f (t) = t^n.$
The case $n = 0$ is trivial.
Otherwise, take $\dt = \ep^2.$
Using selfadjointness at the first step,
using Cauchy-Schwarz at the second step,
and using $-1 \leq a - b \leq 1$ at the fifth step,
we get
\begin{align*}
\| f (a) - f (b) \|_{2, \ta}^2
  & = \ta \big( (a - b) (a - b)^{2 n - 1} \big)
    \leq \ta \big( (a - b)^{2} \big)^{1/2}
              \ta \big( (a - b)^{4 n - 2} \big)^{1/2}
       \\
  & = \| a - b \|_{2, \ta} \big\| (a - b)^{2 n - 1} \big\|_{2, \ta}
    \leq  \| a - b \|_{2, \ta} \big\| (a - b)^{2 n - 1} \big\|
       \\
  & < \dt
    = \ep^2,
\end{align*}
so $\| f (a) - f (b) \|_{2, \ta} < \ep.$
\end{proof}

The following lemma should be known,
but we have not found a reference.

\begin{lem}\label{L:NonAtom}
Let $(X, \mu)$ be a finite measure space which is nonatomic
in the sense that for every measurable set $E \S X$
with $\mu (E) > 0,$
there exists a measurable set $F \S E$
with $0 < \mu (F) < \mu (E).$
Let $E \S X$ be measurable,
and suppose $0 \leq \af \leq \mu (E).$
Then there exists a measurable set $F \subset E$ with $\mu (F) = \af.$
\end{lem}

\begin{proof}
We first claim that for every measurable set $E \S X$
with $\mu (E) > 0,$
there exists a measurable set $F \S E$
with $0 < \mu (F) \leq \frac{1}{2} \mu (E).$
Indeed, choose $F_0 \S E$
with $0 < \mu (F_0) < \mu (E).$
If $\mu (F_0) \leq \frac{1}{2} \mu (E)$ take $F = F_0,$
and otherwise take $F = E \setminus F_0.$

Iterating this argument,
we find that for every measurable set $E \S X$
with $\mu (E) > 0,$
and every $\ep > 0,$
there exists a measurable set $F \S E$
with $0 < \mu (F) < \ep.$

Now let $E \S X$ satisfy $\mu (E) > 0,$
and set
\[
S = \{ \mu (F) \colon {\mbox{$F \S E$ measurable}} \}.
\]
We claim that $S$ is dense in $[0, \, \mu (E)],$
and we prove this by showing that
$S$ is $\ep$-dense in $[0, \, \mu (E)]$ for every $\ep > 0.$
So let $\ep > 0.$
Let ${\mathcal{M}}_0$ be the set of all measurable
subsets $F \S E$ such that $\mu (F) < \ep,$
and let ${\mathcal{M}}$ be the set of all countable unions
of elements in ${\mathcal{M}}_0.$
Set
\[
\bt = \sup ( \{ \mu (B) \colon B \in {\mathcal{M}} \} ).
\]

Our first step to to show that there exists $B \in {\mathcal{M}}$
such that $\mu (B) = \bt.$
Choose $B_n \in {\mathcal{M}}$ such that
$\mu (B_n) > \bt - \frac{1}{n},$
and take $B = \bigcup_{n = 1}^{\infty} B_n.$
Then $B \in {\mathcal{M}},$ so $\mu (B) \leq \bt,$
while $\mu (B) \geq \mu (B_n)  > \bt - \frac{1}{n}$ for all $n.$
So $\mu (B) = \bt.$

Next, if $\mu (B) < \mu (E),$ applying the second claim to
$E \setminus B$ gives $F \S E \setminus B$ such that
$0 < \mu (F) < \ep.$
Then $B \cup F \in {\mathcal{M}}$ and $\mu (B \cup F) > \bt.$
This contradiction shows that $\mu (B) = \bt = \mu (E).$

Accordingly,
there exist measurable sets $F_1, F_2, \ldots \in {\mathcal{M}}_0$
such that $\mu \left( \bigcup_{k = 1}^{\infty} F_k \right) = \mu (E).$
For $n \geq 0$ set $B_n = \bigcup_{k = 1}^{n} F_k.$
Then
\[
0 = \mu (B_0) \leq \mu (B_1) \leq \cdots
\andeqn \limi{n} \mu (B_n) = \mu (E),
\]
while $\mu (B_{n}) - \mu (B_{n - 1}) \leq \mu (E_n) < \ep$
for all $n.$
This implies $\ep$-density of~$S,$ and hence density.

Finally, we prove the result.
If $\af = 0$ take $F = \varnothing.$
Otherwise, apply the previous claim to $E$ to find
$F_1 \S E$ such that $\af - \frac{1}{2} < \mu (F_1) < \af,$
apply the previous claim to $E \setminus F_1$ to find
$F_2 \S E \setminus F_1$ such that
\[
\af - \mu (F_1) - \tfrac{1}{3} < \mu (F_2) < \af - \mu (F_1),
\]
apply the previous claim to $E \setminus (F_1 \cup F_2)$ to find
$F_3 \S E \setminus (F_1 \cup F_2)$ such that
\[
\af - \mu (F_1) - \mu (F_2) - \tfrac{1}{4}
 < \mu (F_3) < \af - \mu (F_1) - \mu (F_2),
\]
etc.
Then set $F = \bigcup_{n = 1}^{\infty} F_n.$
\end{proof}

\begin{lem}\label{L:NoMinPj}
Let $N$ be a von Neumann algebra
with separable predual and no minimal \pj s,
and let $\ta$ be a normal tracial state on~$N.$
Let $p \in N$ be a \pj\  with $\ta (p) > 0,$
and let $0 < \bt < \ta (p).$
Then there exists a \pj\  $q \in N$ with $q \leq p$ such that
$\ta (q) = \bt.$
\end{lem}

\begin{proof}
Suppose that $(N_i)_{i \in I}$ is a family of von Neumann algebras
and the result holds for $N_i$ for each $i \in I.$
We claim that the result then holds for
the von Neumann algebra direct sum $\bigoplus_{i \in I} N_i.$
This is easily
seen by normalizing the restriction of the tracial state to $N_i$
for each $i \in I.$

By type decomposition,
it therefore suffices to prove this separately for von Neumann algebras
of type ${\mathrm{I}}_n$ for fixed $n \in \N \cup \{ \I \},$
for type ${\mathrm{II}}_1,$ for type ${\mathrm{II}}_{\infty},$
and for type~III.
We can ignore types ${\mathrm{I}}_{\infty},$ ${\mathrm{II}}_{\infty},$ and~III,
because they have no normal tracial states.

We consider the case type ${\mathrm{I}}_n$ for fixed $n \in \N.$
Using a direct sum decomposition as at the beginning of the proof,
we may reduce to the case that the \pj~$p$ has constant rank,
say~$m.$
Thus $\ta (p) = m / n.$
We can also assume that there is a finite measure space $(X, \mu)$
such that $N = L^{\I} (X, \mu, M_n).$
Then there is a nonnegative $f \in L^1 (X, \mu)$ such that,
with ${\mathrm{Tr}}$ denoting the standard
(unnormalized) trace on~$M_n,$
we have
\[
\ta (a) = \frac{1}{n} \int_X {\mathrm{Tr}} (a (x)) f (x) \, d \mu (x)
\]
for all $a \in N.$

Define $\nu = \frac{1}{n} f \cdot \mu,$
so that
$\ta (a) = \int_X {\mathrm{Tr}} (a (x)) \, d \nu (x)$
for all $a \in N.$
The assumption that $N$ has no minimal projections
ensures that $\mu$ is nonatomic in the sense of Lemma~\ref{L:NonAtom}.
We claim that the same is true of~$\nu.$
So assume $E \S X$ is measurable and $\nu (E) \neq 0.$
Then there is $\ep > 0$ such that the set
$Y = \{ x \in E \colon f (x) > \ep \}$
satisfies $\mu (Y) > 0.$
Any set $F \S Y$ such that $0 < \mu (F) < \mu (Y)$
then satisfies
$0 < \nu (F) < \nu (Y) \leq \nu (E).$
The claim follows.

As above, our choices imply that $\ta (p) = \frac{m}{n}.$
Also $\nu (X) = 1$ and $n \bt / m = \bt / \ta (p) \leq 1.$
Use Lemma~\ref{L:NonAtom} to choose
a measurable subset $F \S X$ such that $\nu (F) = n \bt / m.$
Then the \pj\  %
\[
q (x) = \begin{cases}
   p (x) & x \in F
       \\
   0     & x \in X \SM F
\end{cases}
\]
satisfies the conclusion.

It remains to consider the type ${\mathrm{II}}_1$ case.
Let $Z = N \cap N'$ be the center of~$N,$
and let $T \colon N \to Z$ be the center valued trace
(Theorem~8.2.8 of~\cite{KR}).
Proposition~8.3.10 of~\cite{KR} implies that $\ta \circ T = \ta.$
Let $a = (\bt / \ta (p) ) T (p) \in Z.$
By Theorem~8.4.4 of~\cite{KR}
(referring to the proof of Theorem~8.4.3 of~\cite{KR}
for the definition of~$\Dt$),
there exists a \pj\  $q_0 \in N$ such that $T (q_0) = a.$
Theorem 8.4.3(vi) of~\cite{KR} implies that $q_0$
is \mvnt\  to a \pj\  $q \leq p.$
Clearly $\ta (q) = \ta (q_0) = \bt.$
\end{proof}

\begin{lem}\label{L:FineRR0}
Let $N$ be a von Neumann algebra
with separable predual and no minimal \pj s,
and let $\ta$ be a normal tracial state on~$N.$
Let $a \in N$ satisfy $0 \leq a \leq 1,$
and let $\ep > 0.$
Then there exist $m \in \N,$
\pj s $p_1, p_2, \ldots, p_m \in N$
such that $\sum_{j = 1}^m p_j = 1$ and $\ta (p_j) < \ep$
for $j = 1, 2, \ldots, m,$
and distinct $\ld_1, \ld_2, \ldots, \ld_m \in [0, 1],$
such that the element $b = \sum_{j = 1}^m \ld_j p_j$
satisfies $\| b - a \| < \ep.$
\end{lem}

\begin{proof}
It follows from Lemma~\ref{L:NoMinPj}
that every \pj\  $q$ can be written as
$q = \sum_{k = 1}^l q_j$ for \pj s $q_j$ with $\ta (q_j) < \ep.$
Since $N$ has real rank zero, we may therefore find
$c \in N$ of the form
$c = \sum_{j = 1}^m \mu_j p_j$ for \pj s $p_1, p_2, \ldots, p_m \in N$
with $\sum_{j = 1}^m p_j = 1$ and $\ta (p_j) < \ep$ for all $j,$
and with $\mu_1, \mu_2, \ldots, \mu_m \in [0, 1],$
such that $\| c - a \| < \frac{1}{2} \ep.$
The only defect is that the $\mu_j$ are not necessarily distinct.
Choose distinct $\ld_1, \ld_2, \ldots, \ld_m \in [0, 1]$
such that $| \ld_j - \mu_j | < \frac{1}{2} \ep$ for all $j,$
and set $b = \sum_{j = 1}^m \ld_j p_j.$
\end{proof}

\begin{prp}\label{P:NoAtomD}
Let $A$ be a separable unital \ca.
Let $\ta$ be a tracial state on $A$
such that, with $\pi_{\ta}$
being the associated Gelfand-Naimark-Segal representation,
the von Neumann algebra $\pi_{\ta} (A)''$ has no minimal \pj s.
Let $S = \{ a \in A \colon 0 \leq a \leq 1 \}.$
For $a \in S$ let $\mu_a$ be the Borel probability measure on $[0, 1]$
defined by $\int_0^1 f \, d \mu_a = \ta (f (a))$
for $f \in C ([0, 1]).$
Then there exists a dense $G_{\dt}$-set $G \S S$
such that, for every $a \in G$ and every $t \in [0, 1],$
we have $\mu_a (\{ t \}) = 0.$
\end{prp}

\begin{proof}
Consider a partition $P = \{ t_0, t_1, \ldots, t_n \}$ of $[0, 1]$
with $0 = t_0 < t_1 < \cdots < t_n = 1.$
We write $\ord (P) = n,$ and call it the order of $P.$
For $k = 1, 2, \ldots, n,$
let $f_k \colon [0, 1] \to [0, 1]$
be the \cfn\  which is is linear on each interval $[t_{j - 1}, \, t_j],$
equal to~$1$ on $[t_{k - 1}, \, t_k],$
and equal to~$0$ on $[0, \, t_{k - 2}]$ and $[t_{k + 1}, \, 1]$
(when these intervals are not empty).
Then set $F_P = \{ f_1, f_2, \ldots, f_n \} \S C ([0, 1]).$
Further let $U_P \S S$ be the open set given by
\[
U_P = \big\{ a \in S \colon
        {\mbox{$\ta (f (a)) < 7 n^{-1}$ for $f \in F_P$}} \}.
\]

For $n \in \N,$ define
\[
V_n = \bigcup_{\ord (P) = n} U_P,
\]
which is an open subset of~$S.$
Let $a \in \bigcap_{n = 1}^{\infty} V_n.$
We claim that $\mu_a ( \{ t \} ) = 0$ for all $t \in [0, 1].$

To prove the claim, let $t \in [0, 1]$ and let $\ep > 0.$
Choose $n \in \N$ such that $7 n^{-1} < \ep.$
Choose a partition $P$ of order $n$ such that $a \in U_P.$
Then there is $f \in F_P$ such that $f (t) = 1.$
So
\[
\mu_a ( \{ t \} )
   \leq \int_0^1 f \, d \mu
   = \ta (f (a))
   < 7 n^{-1}
   < \ep.
\]
Since $\ep > 0$ is arbitrary, this proves the claim.

We now claim that $V_n$ is dense.
Let $a \in S,$ let $n \in \N,$ and let $\ep > 0.$
We must find a partition $P$ of order $n$
and $b \in U_P$ such that $\| a - b \| < \ep.$
It suffices to consider elements $a$ such that there is $\ep_0 > 0$
with $\ep_0 \leq a \leq 1 - \ep_0,$
and we may further reduce the size of $\ep$ and assume $\ep < \ep_0.$
Let $(H_{\ta}, \pi_{\ta}, \xi_{\ta})$
be the Gelfand-Naimark-Segal representation
associated with $\ta.$
We also write $\ta$ for the tracial state on $\pi_{\ta} (A)''.$

Set $\dt = \min \big( \frac{1}{2} \ep, \frac{1}{n} \big).$

Apply Lemma~\ref{L:FineRR0} in $\pi_{\ta} (A)''$ with $a$ as given
and with $\dt$ in place of $\ep,$
finding $c = \sum_{j = 1}^m \ld_j p_j \in \pi_{\ta} (A)''.$
We have $\ep_0 - \dt \leq c \leq 1 - (\ep_0 - \dt).$
\Wolog\  $0 < \ld_1 < \ld_2 < \cdots < \ld_m < 1.$
For $l = 1, 2, \ldots, n - 1,$
choose $k (l)$ such that
\[
\sum_{j = 1}^{k (l)} \ta (p_j)
 \leq \frac{l}{n}
 < \sum_{j = 1}^{k (l) + 1} \ta (p_j).
\]
Then choose $t_l$ such that $\ld_{k (l)} < t_l < \ld_{k (l) + 1}.$
Set $k (0) = 0,$ $k (n) = m,$ $t_0 = 0,$ and $t_n = 1.$
Then $P = \{ t_0, t_1, \ldots, t_n \}$ is a partition of $[0, 1]$
of order~$n.$
We have
\begin{align*}
0 & = t_0 < \ld_{k (0) + 1}
    < \cdots < \ld_{k (1)} < t_1 < \ld_{k (1) + 1}
    < \cdots < \ld_{k (2)} < t_2 < \ld_{k (2) + 1}
          \\
  & < \cdots < \ld_{k (n - 1)} < t_{n - 1} < \ld_{k (n - 1) + 1}
    < \cdots < \ld_{k (n)} < t_n = 1.
\end{align*}
Moreover,
since $\ta (p_{k (l - 1) + 1}) < \dt,$
we have
\[
\sum_{j = 1}^{k (l - 1)} \ta (p_j) > \frac{l}{n} - \dt,
\]
whence by subtraction
\[
\sum_{j = k (l - 1) + 1}^{k (l)} \ta (p_j)
 < \frac{1}{n} + \dt
 \leq \frac{2}{n}.
\]
Since each $f \in F_P$ has support contained in the union
of three of the intervals $[t_{l - 1}, \, t_l],$
it follows that $\ta (f (c)) < \frac{6}{n}$
for all $f \in F_P.$

Use Lemma~\ref{L:FCalcTNorm} to choose $\rh > 0$
such that whenever $B$ is a unital \ca\  %
and $a, b \in A$ satisfy $0 \leq a, b \leq 1$
and $\| a - b \|_{2, \ta} < \rh,$
then $\| f (a) - f (b) \|_{2, \ta} < \frac{1}{n}$
for every $f \in F_P.$
Set
\[
T = \big\{ x \in A_{\sa} \colon
             \| x - a \| \leq \tfrac{1}{2} \ep \big\}.
\]
According to the Kaplansky Density Theorem,
$\pi_{\ta} (T)$ is strong operator dense in
\[
\big\{ y \in \pi_{\ta} (A)'' \colon
    \| y - \pi_{\ta} (a) \| \leq \tfrac{1}{2} \ep \big\}.
\]
In particular, there is $b \in T$ such that
$\| \pi_{\ta} (b) \xi_{\ta} - c \xi_{\ta} \| < \rh.$
Note that, for $y \in \pi_{\ta} (A)'',$
\[
\| y \xi_{\ta} \|^2
 = \langle y \xi_{\ta}, y \xi_{\ta} \rangle
 = \langle y^* y \xi_{\ta}, \xi_{\ta} \rangle
 = \ta (y^* y)
 = \| y \|_{2, \ta}^2.
\]
Thus
$\| \pi_{\ta} (b) - c \|_{2, \ta} < \rh.$
So,
using the choice of $\rh$
and $| \ta (a) | \leq \| a \|_{2, \ta},$
we have
$\big| \ta \big( \pi_{\ta} (f (b)) - f (c) \big) \big| < \frac{1}{n}$
for $f \in F_P.$
It follows that $\ta (f (b)) < \frac{7}{n}$
for $f \in F_P.$
Thus $b \in U_P.$
Since $\| b - a \| < \ep,$ the proof is complete.
\end{proof}

\begin{cor}\label{C:TRPForShift}
Let $A$ be a separable unital \ca.
Let $\ta_0$ be a tracial state on $A$
such that, with $\pi_{\ta_0}$
being the associated Gelfand-Naimark-Segal representation,
the von Neumann algebra $\pi_{\ta_0} (A)''$ has no minimal \pj s.
Let $\sm$ be either the minimal shift or the maximal shift
on $B = \bigotimes_{n \in \Z} A$
(minimal or maximal tensor product, as appropriate),
as in Definition~\ref{D:Shift}.
Then $\sm$ has the \nptrpw\  the infinite tensor product
tracial state $\ta$ on $B$ obtained from $\ta_0.$
\end{cor}

\begin{proof}
Using Proposition~\ref{P:NoAtomD},
we see that the hypotheses of
Proposition~\ref{P:TRPWithNoAtom} are satisfied.
\end{proof}

\section{$Z$-stable C*-algebras with tracial rank zero}\label{Sec:Main}

\indent
We begin by defining a useful metric on the automorphisms
of a separable \ca.

\begin{ntn}\label{N-2910ZMetric}
Let $A$ be a separable \ca.
Let $\Aut (A)$ be the set of all automorphisms of $A.$
For any enumeration $S = (a_1, a_2, \ldots )$ of a countable dense
subset of~$A,$ we define metrics on $\Aut (A)$ by
\[
\rh_S^{(0)} (\af, \bt)
 = \sum_{k = 1}^{\I} 2^{- k}  \| \af (a_k) - \bt (a_k) \|
\andeqn
\rh_S (\af, \bt)
 = \rh_S^{(0)} (\af, \bt) + \rh_S^{(0)} (\af^{-1}, \bt^{-1}).
\]
\end{ntn}

The following result is well known.
We have been unable to find a reference,
so we sketch the proof.

\begin{lem}\label{L:AutTop}
For any $S$ as in Notation~\ref{N-2910ZMetric},
the metrics $\rh_S^{(0)}$ and $\rh_S$
define the topology of pointwise norm convergence on $\Aut (A),$
that is, the topology in which a net $(\af_i)_{i \in I}$ converges
to $\af$ \ifo\  $\| \af_i (a) - \af (a) \| \to 0$ for all $a \in A.$
Moreover, for every such $S,$ the metric $\rh_S$ is complete.
\end{lem}

\begin{proof}
We first prove that $\rh_S^{(0)}$
defines the right topology.
We have to show that if
$(\af_i)_{i \in I}$ is a net in $\Aut (A)$
such that $\rh_S^{(0)} (\af_i, \af) \to 0,$
then $\| \af_i (a) - \af (a) \| \to 0$ for all $a \in A.$
This follows by a standard $\frac{\ep}{3}$~argument from
$\| \af_i (a_k) - \af (a_k) \| \to 0$ for all $k \in \N$
and $\sup_{i \in I} \| \af_i \| \leq 1.$
To complete the proof of the first statement,
it is enough to show that
$\af_i \to \af$ pointwise implies  $\af_i^{-1} \to \af^{-1}$ pointwise.
For $a \in A,$ setting $b = \af^{-1} (a),$ we have
\[
\big\| \af_i^{-1} (a) - \af^{-1} (a) \big\|
 = \big\| \af_i^{-1} \big( \af (b) - \af_i (b) \big) \big\|
 = \| \af (b) - \af_i (b) \|
\to 0.
\]

It remains to prove that $\Aut (A)$ is complete in $\rh_S.$
Let $(\af_n)_{n \in \N}$ be a Cauchy sequence.
Then $(\af_n (a_k))_{n \in \N}$
is a Cauchy sequence for every $k \in \N.$
A standard $\frac{\ep}{3}$~argument
shows that $(\af_n (a))_{n \in \N}$ is Cauchy for all $a \in A.$
So $\af (a) = \limi{n} \af_n (a)$ exists
for all $a \in A.$
Clearly $\af$ is an endomorphism of $A.$
Similarly, $\bt (a) = \limi{n} \af_n^{-1} (a)$ exists
for all $a \in A,$
and $\bt$ is an endomorphism of $A.$

Now let $a \in A,$ and set $b = \bt (a).$
Then
\[
\| \af (\bt (a)) - a \|
  \leq \| \af (b) - \af_n (b) \|
        + \| \af_n \| \cdot \| \bt (a) - \af_n^{-1} (a) \|.
\]
Both terms on the right converge to~$0,$
so $\af ( \bt (a)) = a.$
Thus $\af \circ \bt = \id_A.$
Similarly $\bt \circ \af = \id_A.$
It follows that $\af \in \Aut (A),$ and that $\rh_S (\af_n, \af) \to 0.$
This proves completeness.
\end{proof}

The metric $\rh_S^{(0)}$ is usually not complete,
since in general a sequence of automorphisms
can converge pointwise to an endomorphism which is not surjective.

\begin{ntn}\label{SetNtn}
Let $A$ be a \sfsuca.
For a finite set $F \S A,$ for $\ep > 0,$ for $n \in \N,$
and for a nonzero positive element $x \in A,$
we define $W (F, \ep, n, x) \S \Aut (A)$ to be the set of all
$\af \in \Aut (A)$ such that
there are \mops\  $e_0, e_1, \ldots, e_n \in A$ with:
\begin{enumerate}
\item\label{2906SetNtn-1}
$\| \af (e_j) - e_{j + 1} \| < \ep$ for $j = 0, 1, \ldots, n - 1.$
\item\label{2906SetNtn-2}
$\| e_j a - a e_j \| < \ep$ for $j = 0, 1, \ldots, n$ and all $a \in F.$
\item\label{2906SetNtn-3}
$1 - \sum_{j = 0}^{n} e_j$ is \mvnt\  to a
\pj\  in ${\overline{x A x}}.$
\end{enumerate}
\end{ntn}

\begin{lem}\label{Intersect}
Let $A$ be a separable \sfsuca\  with real rank zero.
Then the set of $\af \in \Aut (A)$ which have the \trp\  is a
countable intersection of sets of the form
$W (F, \ep, n, x).$
\end{lem}

\begin{proof}
Clearly every $\af \in \Aut (A)$ with the \trp\  is in every
$W (F, \ep, n, x).$

Choose a countable dense subset $S \S A,$
and let ${\mathcal{F}}$ be the set of all finite subsets of~$S.$
Also choose a countable set $P$ of nonzero \pj s in $A$ such that
every nonzero \pj\  in $A$ is \mvnt\  to a \pj\  in $P.$
This choice is possible because $A$ is separable and
because any two \pj s $p, q \in A$ with $\| p - q \| < 1$
are necessarily \mvnt.
Then one easily checks that $\af \in \Aut (A)$ has the \trp\  %
\ifo\  %
\[
\af \in \bigcap_{F \in {\mathcal{F}}}
            \bigcap_{m = 1}^{\I} \bigcap_{n = 1}^{\I}
            \bigcap_{p \in P} W \left( F, \tfrac{1}{m}, n, p \right).
\]
This completes the proof.
\end{proof}

\begin{lem}\label{Open}
Let $A$ be a \sfsuca.
For every finite set $F \S A,$  every $\ep > 0,$  every $n \in \N,$
and every nonzero positive element $x \in A,$
the set $W (F, \ep, n, x)$ is open in $\Aut (A).$
\end{lem}

\begin{proof}
Let $\af \in W (F, \ep, n, x).$
Choose \mops\  $e_0, e_1, \ldots, e_n$
in~$A$
such that conditions~(\ref{2906SetNtn-1}),
(\ref{2906SetNtn-2}), and~(\ref{2906SetNtn-3})
in Notation~\ref{SetNtn} hold.
Set
\[
\dt
 = \min \big( \big\{ \ep - \| \af (e_j) - e_{j + 1} \|
                        \colon 0 \leq j \leq n - 1 \big\} \big).
\]
If $\bt \in \Aut (A)$ satisfies $\| \bt (e_j ) - \af (e_j) \| < \dt$
for $j = 0, 1, \ldots, n - 1,$
we immediately get
$\| \bt (e_j) - e_{j + 1} \| < \ep$ for $j = 0, 1, \ldots, n - 1.$
This is the first condition for membership in $W (F, \ep, n, x).$
Since the other two conditions don't mention the automorphism,
we get $\bt \in W (F, \ep, n, x).$
\end{proof}

The previous two results show that the set of automorphisms
with the \trp\  is a $G_{\dt}$-set in $\Aut (A).$
We next show that,
for suitable \ca s $A,$
this set is dense.
This is harder.

We recall the Jiang-Su algebra~$Z,$
from Theorem~2.9 of~\cite{JS}.
It is a simple separable unital \ca\  which is not of type~$I$
and which has a unique tracial state $\ta_0.$

\begin{thm}\label{T:PrpOfZ}
The Jiang-Su algebra~$Z$ has the following properties:
\begin{enumerate}
\item\label{StSelfAbs}
The \hm\  $\io_Z \colon Z \to Z \otimes Z,$
given by $\io (a) = 1 \otimes a,$
is approximately unitarily equivalent to an isomorphism.
\item\label{IsoInfTP}
The infinite tensor product $\bigotimes_{n \in \Z} Z$ is
isomorphic to~$Z.$
\item\label{Abs}
Let $A$ be a simple separable infinite dimensional unital \ca\  with
tracial rank zero in the sense of \cite{LnTAF} and~\cite{LnTTR},
and which satisfies the Universal Coefficient Theorem
(Theorem~1.17 of~\cite{RSUCT}).
Then $Z \otimes A \cong A.$
\end{enumerate}
\end{thm}

\begin{proof}
For~(\ref{StSelfAbs}),
combine $Z \otimes Z \cong Z$
(Theorem~8.7 of~\cite{JS})
with the fact that any two unital endomorphisms of~$Z$
are approximately unitarily equivalent
(a consequence of Theorem~7.6 of~\cite{JS}).
Part~(\ref{IsoInfTP}) is Corollary~8.8 of~\cite{JS}
(with different indexing).

For~(\ref{Abs}),
note that Proposition~3.7 of~\cite{PhtRp2}
(proved using Lin's classification theorem,
Theorem~5.2 of~\cite{Ln15}) implies that
$A$ is a simple AH~algebra with real rank zero
and no dimension growth.
Now Theorem~2.1 of~\cite{EGL1}
implies that $A$ is approximately divisible,
so that $Z \otimes A \cong A$ by Theorem~2.3 of~\cite{TW}.
\end{proof}

\begin{rmk}
The \ca s covered by Theorem~\ref{T:PrpOfZ}(\ref{Abs})
are exactly the simple unital AH~algebras with real rank zero
and no dimension growth, as classified in~\cite{EGL2}.
One can see this from the proofs of the results
used in the proof of Theorem~\ref{T:PrpOfZ}(\ref{Abs}).
\end{rmk}

\begin{cor}\label{C:Z}
There exists $\ph \in \Aut (Z)$ which has the \nptrpw~$T (Z).$
\end{cor}

\begin{proof}
Combine Theorem~\ref{T:PrpOfZ}(\ref{IsoInfTP})
and Corollary~\ref{C:TRPForShift}.
\end{proof}

\begin{lem}\label{L:TRMForInnPtb}
Let $A$ be an \idsfsuca,
and let $\af \in \Aut (A)$ have the \trp.
Let $u \in A$ be unitary.
Then $\Ad (u) \circ \af$ has the \trp.
\end{lem}

\begin{proof}
Set $\bt = \Ad (u) \circ \af.$
Let $S \S A$ be a finite set, let $\ep > 0,$ let $n \in \N,$
and let $x \in A$ be a nonzero positive element.
Apply the \trp\  for $\af$ with $S \cup \{ u \}$ in place of $S,$
with $\tfrac{1}{2} \ep$ in place of $\ep,$
and with $n$ and $x$ as given.
Let $e_0, e_1, \ldots, e_n$ be the resulting \pj s.
We claim that these verify the requirements of the
definition for $S,$ $\ep,$ $n,$ and $x.$
We need only check that $\| \bt (e_j) - e_{j + 1} \| < \ep$
for $j = 0, 1, \ldots, n - 1.$
We have
\begin{align*}
\| \bt (e_j) - e_{j + 1} \|
& = \| u \af (e_j) u^* - e_{j + 1} \|
  \leq \|\af (e_j) - e_{j + 1} \|
        + \| u e_{j + 1} u^* - e_{j + 1} \|
               \\
& < \tfrac{1}{2} \ep + \tfrac{1}{2} \ep
  = \ep.
\end{align*}
This completes the proof.
\end{proof}

\begin{prp}\label{P:TRPDense}
Let $A$ be a simple separable unital \ca\  with tracial
rank zero and such that $Z \otimes A \cong A.$
Let $\af \in \Aut (A).$
Then for every finite set $F \S A$ and every $\ep > 0,$
there exists $\bt \in \Aut (A)$ such that:
\begin{enumerate}
\item\label{Dense:TRP}
$\bt$ has the \trp.
\item\label{Dense:App}
$\| \bt (a) - \af (a) \| < \ep$ for all $a \in F.$
\item\label{Dense:AUE}
$\bt$ is approximately unitarily equivalent to $\af.$
\end{enumerate}
\end{prp}

\begin{proof}
\Wolog\  $\| a \| \leq 1$ for all $a \in F.$

For any unital \ca\  $B,$
define $\io_B \colon B \to Z \otimes B$
by $\io (b) = 1 \otimes b$ for $b \in B.$
Using Theorem~\ref{T:PrpOfZ}(\ref{StSelfAbs})
and tensoring with $\id_A,$
we see that
$\io_{Z \otimes A} \colon Z \otimes A \to Z \otimes Z \otimes A$
is approximately unitarily equivalent to an isomorphism.
Since $Z \otimes A \cong A$ by hypothesis,
it follows that there is an isomorphism $\gm \colon A \to Z \otimes A$
which is approximately unitarily equivalent to $\io_A.$
In particular,
there is sequence $(u_n)_{n \in \N}$
of unitaries in $Z \otimes A$ such that
$\limi{n} \| u_n (1 \otimes a) u_n^* - \gm (a) \| = 0$
for all $a \in A.$

Use Corollary~\ref{C:Z} to find $\ph \in \Aut (Z)$
with the \nptrpw~$T (Z).$
For $n \in \N,$ define a unitary $v_n \in A$ by
$v_n = \gm^{-1} \big( u_n \cdot (\ph \otimes \af) (u_n^*) \big).$
Then set
$\bt_n = \Ad (v_n) \circ \gm^{-1} \circ (\ph \otimes \af) \circ \gm.$

We prove part~(\ref{Dense:TRP}) for $\bt_n$ for all~$n.$
We begin by proving that
$\ph \otimes \af$ has the \nptrpw\  $T (Z \otimes A)$
(Definition~\ref{D:NPjTRP}).
It suffices to consider finite sets consisting of
elementary tensors.
Thus,
let $S \S Z$ and $T \S A$ be finite, let $\ep > 0,$ and let $n \in \N.$
We will verify the conditions of Definition~\ref{D:NPjTRP}
with
\[
\{ x \otimes a \colon {\mbox{$x \in S$ and $a \in T$}} \}
\]
in place of $S,$ and with $\ep$ and $n$ as given.
\Wolog\  $\| a \| \leq 1$ for all $a \in T.$
Let $\ta_0$ be the unique tracial state on $Z.$
Because $\ph$ has the \nptrpw\  $\{ \ta_0 \},$
there exist $d_0, d_1, \ldots, d_n \in Z$ such that:
\begin{enumerate}
\item\label{Pf:Pos}
$0 \leq d_j \leq 1$ for $j = 0, 1, \ldots, n.$
\item\label{Pf:Orth}
$d_j d_k = 0$ for $j, k = 0, 1, \ldots, n$ with $j \neq k.$
\item\label{Pf:Shift}
$\| \ph (d_j) - d_{j + 1} \| < \ep$ for $j = 0, 1, \ldots, n - 1.$
\item\label{Pf:Comm}
$\| d_j x - x d_j \| < \ep$ for $j = 0, 1, \ldots, n$ and all $x \in S.$
\item\label{Pf:Rem}
With $d = \sum_{j = 0}^{n} d_j,$ we have $\ta_0 (1 - d) < \ep.$
\end{enumerate}
Set $c_j = d_j \otimes 1$ for $j = 0, 1, \ldots, n.$
The first four conditions of Definition~\ref{D:NPjTRP},
which are the analogs of the first four conditions above,
are immediate.
For the last part,
let $\ta \in T (Z \otimes A),$
and use Lemma~2.12 of~\cite{JS} to write $\ta = \ta_0 \otimes \sm$
for some $\sm \in T (A).$
Then
\[
\ta \left( 1 - \sssum{j = 0}{n} c_j \right)
  = \ta_0 \left( 1 - \sssum{j = 0}{n} d_j \right) \sm (1)
  < \ep.
\]
So $\ph \otimes \af$ has the \nptrpw~$T (A).$

Since $Z \otimes A$ has tracial rank zero (being isomorphic to $A$),
Proposition~\ref{P:NPImpTRP} implies that
$\ph \otimes \af$ has the \trp.
Therefore so does $\gm^{-1} \circ (\ph \otimes \af) \circ \gm.$
So $\bt_n$ has the \trp\  by Lemma~\ref{L:TRMForInnPtb},
as desired.

Now let $a \in A.$
Then
\[
\bt_n (a)
 = \gm^{-1} \big( u_n \cdot (\ph \otimes \af)
    \big( u_n^* \gm (a) u_n \big) \cdot u_n^* \big).
\]
So
\begin{align*}
\| \bt_n (a) - \af (a) \|
 & \leq \| u_n^* \gm (a) u_n - 1 \otimes a \big\|
           + \big\| \gm^{-1} \big( u_n \cdot (\ph \otimes \af)
                            (1 \otimes a) \cdot u_n^* \big)
                   - \af (a) \big\|
          \\
 & = \| \gm (a) - u_n (1 \otimes a) u_n^* \|
       + \| u_n (1 \otimes \af (a)) u_n^* - \gm (\af (a)) \|.
\end{align*}
This expression converges to zero as $n \to \I.$
Substituting the definition of~$\bt_n,$ we get
\[
\limi{n} \big\|
 \big( \Ad (v_n) \circ \gm^{-1} \circ (\ph \otimes \af) \circ \gm \big) (a)
      - \af (a) \big\|
   = 0.
\]
Since $a \in A$ is arbitrary,
we have shown that
$\gm^{-1} \circ (\ph \otimes \af) \circ \gm$ is
approximately unitarily equivalent to $\af.$
It follows that $\bt_n$ is
approximately unitarily equivalent to $\af$ for all $n.$
This is part~(\ref{Dense:AUE}) for~$\bt_n.$

Now choose $n \in \N$ so large that
$\| u_n (1 \otimes a) u_n^* - \gm (a) \| < \frac{1}{2} \ep$
for all $a \in F \cup \af (F),$
and set $\bt = \bt_n.$
The estimate in the previous paragraph shows that
part~(\ref{Dense:App}) of the conclusion holds,
completing the proof.
\end{proof}

\begin{thm}\label{Main}
Let $A$ be a simple separable unital \ca\  with tracial
rank zero and such that $Z \otimes A \cong A.$
Then there is a dense $G_{\dt}$-set $G \S \Aut (A)$
such that every $\af \in G$ has the \trp.
\end{thm}

\begin{proof}
Let
\[
G = \{ \af \in \Aut (A) \colon {\mbox{$\af$ has the \trp}} \}.
\]
It follows from Lemmas \ref{Intersect} and~\ref{Open}
that $G$ is a $G_{\dt}$-set,
and from Proposition~\ref{P:TRPDense} that $G$ is dense.
\end{proof}

\section{C*-algebras with tracial rank zero which are not
 $Z$-stable}\label{Sec:Old}

It seems to be unknown whether every simple separable
unital, but not necessarily nuclear, \ca~$A$
with tracial rank zero
satisfies $Z \otimes A \cong A,$
even if one also assumes that
$A$ satisfies the Universal Coefficient Theorem.
Even if this is not true,
we can still show that the automorphisms with the \trp\  are
generic among all approximately inner automorphisms.
The obstruction to showing that they are
generic among all automorphisms
is that, as far as we know,
it might be possible for the approximate unitary equivalence class
of some automorphism to contain no automorphism at all
with the \trp.

The proof uses a much weaker condition,
namely tracial approximate divisibility.
Writing the proof in terms of tracial approximate divisibility
makes the ideas clearer and the formulas simpler.
We therefore define tracial approximate divisibility
and develop enough of its basic theory to use it in the proof.
There are no surprises, so the reader not interested
in the details could skip from Definition~\ref{D-2902TrAppDiv}
directly to Notation~\ref{N-2907Inn}.
(In fact, our proof, in Proposition~\ref{P-2907TRZTAD},
that tracial rank zero implies tracial approximate divisibility
also implies the conclusion of Corollary~\ref{C-2907-210}.)
We do take care to give a definition of tracial approximate divisibility
which seems likely to be appropriate for use with infinite \ca s.

\begin{dfn}[Definition~1.1 of~\cite{BKR}]\label{D-2902cncfd}
A finite dimensional C*-algebra is
{\emph{completely noncommutative}}
if it has no commutative direct summands,
equivalently, no abelian central projections.
\end{dfn}

The following is a modification of Definition~1.1 of~\cite{BKR},
following the pattern used
for the definition of tracial rank zero
(Definition~2.1 of~\cite{LnTAF}).
This concept, with a formally stronger definition
(requiring arbitrarily large matrix sizes
and a single summand)
has been considered independently in
unpublished work of Ilan Hirshberg,
who showed that such algebras have stable rank one.

\begin{dfn}\label{D-2902TrAppDiv}
Let $A$ be a simple separable infinite dimensional unital C*-algebra.
We say that $A$ is {\emph{tracially approximately divisible}}
if for every $\ep > 0,$ every $n \in \N,$
every $a_1, a_2, \ldots, a_n \in A,$
and every $y \in A_{+}$ with $\| y \| = 1,$
there exist a \pj~$e \in A,$
a completely noncommutative finite dimensional C*-algebra~$D,$
and an injective unital \hm\  $\ph \colon D \to e A e,$
such that:
\begin{enumerate}
\item\label{D-2902TrAppDiv-AC}
$\| \ph (b) a_k - a_k \ph (b) \| < \ep$ for $k = 1, 2, \ldots, n$
and all $b \in D$ with $\| b \| \leq 1.$
\item\label{D-2902TrAppDiv-Sm}
$1 - e$ is \mvnt\  to a \pj\  in ${\overline{y A y}}.$
\item\label{D-2902TrAppDiv-Norm}
$\| e y e \| > 1 - \ep.$
\end{enumerate}
\end{dfn}

The following version is immediately seen to be equivalent
and is sometimes technically more convenient.
(Note that we do not require the \hm~$\ph$ to be injective.
However, as long as $\ep < 1,$
we must have $\ph (0, 1) \neq 0.$)

\begin{rmk}\label{R-2902UnitalTAD}
Let $A$ be a simple separable infinite dimensional unital C*-algebra.
Then $A$ is tracially approximately divisible
\ifo\  for every $\ep > 0,$ every $n \in \N,$
every $a_1, a_2, \ldots, a_n \in A,$
and every $y \in A_{+}$ with $\| y \| = 1,$
there exists
a completely noncommutative finite dimensional C*-algebra~$D$
and a unital \hm\  $\ph \colon \C \oplus D \to A$
such that:
\begin{enumerate}
\item\label{D-2902UnitalTAD-AC}
$\| \ph (b) a_k - a_k \ph (b) \| < \ep$ for $k = 1, 2, \ldots, n$
and all $b \in \C \oplus D$ with $\| b \| \leq 1.$
\item\label{D-2902UnitalTAD-Sm}
$\ph (1, 0)$ is \mvnt\  to a \pj\  in ${\overline{y A y}}.$
\item\label{D-2902UnitalTAD-Norm}
$\| \ph (0, 1) y \ph (0, 1) \| > 1 - \ep.$
\end{enumerate}
\end{rmk}

\begin{rmk}\label{R-2902NotN1}
In both Definition~\ref{D-2902TrAppDiv}
and Remark~\ref{R-2902UnitalTAD},
we can replace the requirement that $\| y \| = 1$
with the requirement $y \neq 0,$
and the last condition with
$\| e y e \| > \| y \| - \ep$ (in Definition~\ref{D-2902TrAppDiv})
or $\| \ph (0, 1) y \ph (0, 1) \| > \| y \| - \ep$
(in Remark~\ref{R-2902UnitalTAD}).
\end{rmk}

\begin{lem}\label{L-2902StdCncfd}
In Definition~\ref{D-2902TrAppDiv},
Remark~\ref{R-2902UnitalTAD},
and Remark~\ref{R-2902NotN1},
we may restrict $D$ to being one of the three
standard completely noncommutative finite dimensional C*-algebras
$M_2,$ $M_3,$ and $M_2 \oplus M_3$ of Definition~2.6 of~\cite{BKR}.
\end{lem}

\begin{proof}
See the proof (given before the proposition)
of Proposition~2.7 of~\cite{BKR}.
\end{proof}

In fact, we can almost certainly restrict to $D = M_2.$
See Remark~\ref{R-2914StComp} below.

\begin{prp}\label{P-2902SP}
Let $A$ be a simple separable infinite dimensional unital C*-algebra
which is tracially approximately divisible.
Then $A$ has Property~(SP).
\end{prp}

\begin{proof}
Let $y \in A_{+} \SM \{ 0 \}.$
We show that ${\overline{y A y}}$ contains a nonzero \pj.
\Wolog\  $\| y \| = 1.$

Suppose that for every $\ep > 0,$ every $n \in \N,$
and every $a_1, a_2, \ldots, a_n \in A,$
there exists
a completely noncommutative finite dimensional C*-algebra~$D$
and an injective unital \hm\  $\ps \colon D \to A$
such that
$\| \ps (b) a_k - a_k \ps (b) \| < \ep$ for $k = 1, 2, \ldots, n$
and all $b \in D$ with $\| b \| \leq 1.$
Then $A$ is approximately divisible,
so Theorem 1.3(b) of~\cite{BKR} implies that
$A$ has Property~(SP).
In particular,
${\overline{y A y}}$ contains a nonzero \pj.

Otherwise,
there are $\ep > 0,$ $n \in \N,$
and $a_1, a_2, \ldots, a_n \in A$
such that no \hm~$\ps$ as above exists.
By definition,
there nevertheless exist a \pj~$e \in A,$
a completely noncommutative finite dimensional C*-algebra~$D,$
and an injective unital \hm\  $\ph \colon D \to e A e,$
such that
$\| \ph (b) a_k - a_k \ph (b) \| < \ep$ for $k = 1, 2, \ldots, n$
and all $b \in D$ with $\| b \| \leq 1,$
and such that
$1 - e$ is \mvnt\  to a \pj\  $f \in {\overline{y A y}}.$
Clearly $e \neq 1,$
so $f \neq 0.$
Thus in this case also ${\overline{y A y}}$ contains a nonzero \pj.
\end{proof}

When $A$ is finite,
we can omit condition~(\ref{D-2902TrAppDiv-Norm})
in Definition~\ref{D-2902TrAppDiv}.

\begin{lem}\label{L-2907FinNoy}
Let $A$ be a finite simple separable infinite dimensional
unital C*-algebra.
Assume that for every $\ep > 0,$ every $n \in \N,$
every $a_1, a_2, \ldots, a_n \in A,$
and every $y \in A_{+} \SM \{ 0 \},$
there exist a \pj~$e \in A,$
a completely noncommutative finite dimensional C*-algebra~$D,$
and an injective unital \hm\  $\ph \colon D \to e A e,$
such that:
\begin{enumerate}
\item\label{D-2907FinNoy-AC}
$\| \ph (b) a_k - a_k \ph (b) \| < \ep$ for $k = 1, 2, \ldots, n$
and all $b \in D$ with $\| b \| \leq 1.$
\item\label{D-2907FinNoy-Sm}
$1 - e$ is \mvnt\  to a \pj\  in ${\overline{y A y}}.$
\end{enumerate}
Then $A$ is tracially approximately divisible.
\end{lem}

\begin{proof}
Since $A$ has Property~(SP) (Proposition~\ref{P-2902SP}),
the proof is the same as the Property~(SP) case of
the proof of Lemma 1.16 of~\cite{PhT1}.
\end{proof}

We get the following analog of Lemma~2.8 of~\cite{BKR}.

\begin{lem}\label{L-2902BKR28}
Let $A$ be a simple separable infinite dimensional unital C*-algebra
which is tracially approximately divisible,
let $\ep > 0,$ let $n \in \N,$
let $a_1, a_2, \ldots, a_n \in A,$
and let $y \in A_{+} \SM \{ 0 \}.$
Let $C \S A$ be a finite dimensional subalgebra of~$A.$
Then there exists
a standard completely noncommutative finite dimensional C*-algebra~$D$
(as in Lemma~\ref{L-2902StdCncfd})
and a unital \hm\  $\ph \colon \C \oplus D \to A$
such that:
\begin{enumerate}
\item\label{L-2902BKR28-AC}
$\| \ph (b) a_k - a_k \ph (b) \| < \ep$ for $k = 1, 2, \ldots, n$
and all $b \in \C \oplus D$ with $\| b \| \leq 1.$
\item\label{L-2902BKR28-ExC}
$\ph (b) c = c \ph (b)$ for all $c \in C$ and
$b \in \C \oplus D.$
\item\label{L-2902BKR28-Sm}
$\ph (1, 0)$ is \mvnt\  to a \pj\  in ${\overline{y A y}}.$
\item\label{L-2902BKR28-Norm}
$\| \ph (0, 1) y \ph (0, 1) \| > \| y \| - \ep.$
\end{enumerate}
\end{lem}

\begin{proof}
The proof is similar to that of Lemma~2.8 of~\cite{BKR},
but there are extra steps.
Instead of the
three standard completely noncommutative finite dimensional C*-algebras
considered in~\cite{BKR},
there are now six algebras to consider,
namely,
$M_2,$ $M_3,$ $M_2 \oplus M_3,$
$\C \oplus M_2,$ $\C \oplus M_3,$ and $\C \oplus M_2 \oplus M_3.$

\Wolog\  $\ep < 1.$
We set
\[
\ep_1 = \min \left( \frac{\ep}{84 \cdot \max
       \big( \| a_1 \|, \, \| a_2 \|, \, \ldots, \| a_n \| \big)},
     \, \frac{\ep}{3} \right).
\]
We use $84$ rather than $78$
(also $14$ rather than $13,$
and $42$ in place of $39$ in the definition of~$\dt$),
since the largest possible dimension of an image of $\C \oplus D$
is now $14$ rather than~$13.$
We follow the proof of Lemma~2.8 of~\cite{BKR},
starting by finding
a standard completely noncommutative finite dimensional C*-algebra~$D$
and a unital \hm\  $\ph_0 \colon \C \oplus D \to A$
satisfying~(\ref{L-2902BKR28-AC})
with
\[
\dt = \min
 \left( \frac{\ep}{42}, \, \frac{\ep_1}{d}, \, \frac{\ep_3}{d} \right)
\]
in place of~$\ep$
and with $\{ a_1, a_2, \ldots, a_n \}$ expanded to include a
system of matrix units for~$C,$
satisfying~(\ref{L-2902BKR28-Sm}),
satisfying~(\ref{L-2902BKR28-Norm}) with $\dt$ in place of~$\ep,$
and such that $\ph_0 |_{0 \oplus D}$ is injective.
If $\ph_0 (1, 0) = 0,$
then the argument proceeds as in the proof of Lemma~2.8 of~\cite{BKR}.
Otherwise,
proceed as in the proof of Lemma~2.8 of~\cite{BKR}
but using the algebra $\C \oplus D$ in place of $D$ there.
This yields a unital \hm\  $\ph \colon \C \oplus D \to A$
satisfying (\ref{L-2902BKR28-AC}) and~(\ref{L-2902BKR28-ExC}).
Inspection of the second half of the proof
shows that there is moreover a
system~$S$ of matrix units for $\C \oplus D$
such that
$\| \ph (e) - \ph_0 (e) \| < \min (\ep_1, \ep_3) + \ep_1$
for all $s \in S.$
It follows that
\[
\| \ph (1, 0) - \ph_0 (1, 0) \|
  < \min (\ep_1, \ep_3) + \ep_1
  \leq 2 \ep_1
  < 1.
\]
Therefore $\ph (1, 0)$ is \mvnt\  $\ph_0 (1, 0),$
and is thus also \mvnt\  to a \pj\  in ${\overline{y A y}}.$
Moreover,
\[
\| \ph (0, 1) - \ph_0 (0, 1) \|
 = \big\| [1 - \ph (1, 0)] - [1 - \ph_0 (1, 0)] \big\|
 = \| \ph (1, 0) - \ph_0 (1, 0) \|
 < 2 \ep_1.
\]
Therefore
\begin{align*}
\| \ph (0, 1) y \ph (0, 1) \|
& > \| \ph_0 (0, 1) y \ph_0 (0, 1) \|
               - 2 \| \ph (0, 1) - \ph_0 (0, 1) \|
           \\
& > \| y \| - \dt - 2 \ep_1
  \geq \| y \| - \ep,
\end{align*}
as required.
\end{proof}

The following proposition
is the useful substitute for Theorem 1.3(a) of~\cite{BKR}.

\begin{prp}\label{P-2902BKR-13a}
Let $A$ be a simple separable infinite dimensional unital C*-algebra
which is tracially approximately divisible,
let $\ep > 0,$
and let $y \in A_{+}$ satisfy $\| y \| = 1.$
Then there exists a strictly increasing sequence
$(A_m)_{m \in \Nz}$ of unital subalgebras of~$A,$
a sequence $(D_m)_{m \in \N}$
of completely noncommutative finite dimensional C*-algebras,
a sequence $(\ph_m)_{m \in \N}$
of unital \hm s $\ph_m \colon \C \oplus D_m \to A_{m},$
and a sequence $(p_m)_{m \in \N}$ of \pj s in~$A$
such that:
\begin{enumerate}
\item\label{P-2902BKR-13a-DU}
$A = {\ov{\bigcup_{m = 0}^{\I} A_m}}.$
\item\label{P-2902BKR-13a-Comm}
$\ph_m (b) a = a \ph_m (b)$ for all $m \in \N,$
all $a \in A_{m - 1},$ and all $b \in \C \oplus D_m.$
\item\label{P-2902BKR-13a-PjInAn}
$p_m \in A_{m}$ for all $m \in \N.$
\item\label{P-2902BKR-13a-pC}
$p_m \ph_k (b) = \ph_k (b) p_m$
for $m \in \N,$ for $k = 1, 2, \ldots, m,$
and for $b \in \C \oplus D_k.$
\item\label{P-2902BKR-13a-Dom}
$\ph_k (1, 0) \leq p_m$ for $m \in \N$ and $k = 1, 2, \ldots, m.$
\item\label{P-2902BKR-13a-Sm}
$p_m$ is \mvnt\  to a \pj\  in ${\overline{y A y}}.$
\item\label{P-2902BKR-13a-Norm}
$\| (1 - p_m) y (1 - p_m) \| > 1 - \ep$
for all $m \in \N.$
\end{enumerate}
\end{prp}

\begin{proof}
Since $A$ has Property~(SP) (Proposition~\ref{P-2902SP}),
Lemma~1.10 of~\cite{PhT1}
gives orthogonal \nzp s $f_1, g_1 \in {\overline{y A y}}.$
Repeated applications of this lemma give
orthogonal \nzp s $f_2, g_2 \in g_1 A g_1,$
orthogonal \nzp s $f_3, g_3 \in g_2 A g_2,$
etc.
We thus get orthogonal \nzp s
$f_1, f_2, \ldots \in {\overline{y A y}}.$

\Wolog\  $\ep < 1.$

Choose $a_1, a_2, \ldots \in A$
such that $\{ a_1, a_2, \ldots \}$ is dense.
Apply Lemma~\ref{L-2902BKR28}
inductively to get
standard completely noncommutative finite dimensional C*-algebras
$D_1, D_2, \ldots,$
finite dimensional unital subalgebras $C_0 \S C_1 \S \cdots \S A$
(finite dimensionality is justified afterwards),
positive elements $y_0, y_1, \ldots \in A,$
and unital \hm s  $\ph_m \colon \C \oplus D_m \to A$
for $m \in \N,$
such that:
\begin{enumerate}
\item\label{L-2907Pf-AC}
$\| \ph_m (b) a_k - a_k \ph_m (b) \| < 2^{- m}$
for $k = 1, 2, \ldots, m$
and all $b \in \C \oplus D_m$ with $\| b \| \leq 1.$
\item\label{L-2907Pf-CDef}
$C_0 = \C \cdot 1$ and $C_m$ is the \ca\  generated
by $C_{m - 1}$ and $\ph_{m} (\C \oplus D_{m})$ for $m \geq 1.$
\item\label{L-2907Pf-ExC}
For $m \in \N,$ we have
$\ph_m (b) c = c \ph_m (b)$ for all
$c \in C_{m - 1}$ and
$b \in \C \oplus D_m.$
\item\label{L-2907Pf-Sm}
$\ph_m (1, 0)$ is \mvnt\  to a \pj\  in ${\overline{f_m A f_m}}.$
\item\label{L-2907Pf-ydef}
$y_0 = y$ and $y_m = \ph_{m} (0, 1) y_{m - 1} \ph_{m} (0, 1)$
for $m \geq 1.$
\item\label{L-2907Pf-Norm}
$\| y_m \| > \| y_{m - 1} \| - 2^{- m} \ep.$
\setcounter{TmpEnumi}{\value{enumi}}
\end{enumerate}
Using (\ref{L-2907Pf-CDef}) and~(\ref{L-2907Pf-ExC}),
we see that $C_0$ is finite dimensional
and that $C_m$ is finite dimensional
whenever $C_{m - 1}$ is finite dimensional.
Thus, $C_m$ is in fact finite dimensional for all $m \in \N.$

Following the proof of Theorem 1.3(a) of~\cite{BKR}
(which is given after Corollary~2.9 of~\cite{BKR}),
for $m \in \Nz$
let $A_m \S A$ be the subalgebra
\[
A_m = \big\{ a \in A \colon {\mbox{$a \ph_l (b) = \ph_l (b) a$
   for $l = m + 1, \, m + 2, \ldots$ and $b \in \C \oplus D_l$}} \big\}.
\]
Then,
as in~\cite{BKR},
the sequence $(A_m)_{m \in \N}$ is
nondecreasing,
\begin{equation}\label{Eq:2907Star}
\ph_m (\C \oplus D_m) \subset
 \big\{ a \in A_{m} \colon
    {\mbox{$a x = x a$ for all $x \in A_{m - 1}$}} \big\}
\end{equation}
for $m \in \N,$
and ${\ov{\bigcup_{m = 0}^{\I} A_m}} = A.$

We now have $(A_m)_{m \in \Nz}$ and $(\ph_m)_{m \in \N}$ as described,
except that we only know that $(A_m)_{m \in \Nz}$
is nondecreasing,
and we have verified conditions (\ref{P-2902BKR-13a-DU})
and~(\ref{P-2902BKR-13a-Comm})
of the statement.

Define $q_m \in A$ for $m \in \Nz$ inductively by $q_0 = 1$
and $q_m = \ph_m (0, 1) q_{m - 1}$ for $m \in \N.$
We claim that for all $m \in \N$ we have:
\begin{enumerate}
\setcounter{enumi}{\value{TmpEnumi}}
\item\label{P-2902BKR-13a-qPj}
$q_m$ is a \pj.
\item\label{P-2902BKR-13a-qAm1}
$q_m \in A_{m}.$
\item\label{P-2902BKR-13a-qleq}
$q_m \leq \ph_m (0, 1).$
\item\label{P-2902BKR-13a-qlqm}
$q_m \leq q_{m - 1}.$
\item\label{P-2902BKR-13a-qCent}
$q_m$ is in the center of~$C_{m}.$
\item\label{P-2902BKR-13a-qSubeq}
$1 - q_m
   \precsim \ph_1 (1, 0) \oplus \ph_2 (1, 0) \oplus \cdots
      \oplus \ph_m (1, 0).$
\item\label{P-2902BKR-13a-q12}
$q_m y q_m = y_{m}.$
\end{enumerate}
We prove these statements by induction on~$m.$
They are all easy for $m = 1$
(using (\ref{Eq:2907Star}) to get~(\ref{P-2902BKR-13a-qAm1}),
centrality of $(1, 0)$ in $\C \oplus D_m$
to get~(\ref{P-2902BKR-13a-qCent}),
and (\ref{L-2907Pf-ydef}) to get~(\ref{P-2902BKR-13a-q12})).
So assume they hold for $m - 1$;
we prove them for~$m.$

Since $q_{m - 1} \in C_{m - 1},$
it follows from~(\ref{L-2907Pf-ExC})
that $q_{m - 1}$ commutes with $\ph_m (0, 1).$
Therefore $q_m$ is a \pj\  with $q_m \leq \ph_m (0, 1)$
and $q_m \leq q_{m - 1}.$
This is (\ref{P-2902BKR-13a-qPj}),
(\ref{P-2902BKR-13a-qleq}),
and~(\ref{P-2902BKR-13a-qlqm}) for~$m.$
Also,
\[
1 - q_m
 = 1 - q_{m - 1} + q_{m - 1} \ph_m (1, 0)
 \precsim (1 - q_{m - 1}) \oplus \ph_m (1, 0).
\]
Together with (\ref{P-2902BKR-13a-qSubeq}) for $m - 1,$
this gives (\ref{P-2902BKR-13a-qSubeq}) for~$m.$
For~(\ref{P-2902BKR-13a-qAm1}),
we have $q_m \in A_{m}$ because $q_{m - 1} \in A_{m - 1} \S A_{m}$
and $\ph_m (0, 1) \in A_{m}.$

For~(\ref{P-2902BKR-13a-q12}),
we have already seen that $q_{m - 1}$ commutes with $\ph_m (0, 1).$
Using this at the second step,
(\ref{P-2902BKR-13a-q12})~for $m - 1$ at the third step,
and (\ref{L-2907Pf-ydef}) at the fourth step,
we get
\begin{align*}
q_m y q_m
& = \ph_m (0, 1) q_{m - 1} y \ph_m (0, 1) q_{m - 1}
  \\
& = \ph_m (0, 1) q_{m - 1} y q_{m - 1} \ph_m (0, 1)
  = \ph_m (0, 1) y_{m - 1} \ph_m (0, 1)
  = y_{m}.
\end{align*}
This is (\ref{P-2902BKR-13a-q12}) for~$m.$

Finally, we prove (\ref{P-2902BKR-13a-qCent}) for~$m.$
The \pj\  $q_{m - 1}$ is in $C_{m - 1}$
and commutes with all elements of $C_{m - 1}$
by (\ref{P-2902BKR-13a-qCent}) for $m - 1.$
Also, $q_{m - 1}$ commutes with all elements of $\ph_m (\C \oplus D_m)$
by~(\ref{L-2907Pf-ExC}).
Furthermore,
trivially $\ph_m (0, 1)$
is in $\ph_m (\C \oplus D_m)$
and commutes with all elements of $\ph_m (\C \oplus D_m).$
Also, $\ph_m (0, 1)$ commutes with all elements of $C_{m - 1}$
by~(\ref{L-2907Pf-ExC}).
Thus
\[
q_m = q_{m - 1} \ph_m (0, 1)
    \in C^* (C_{m - 1}, \, \ph_m (\C \oplus D_m))
    = C_{m}
\]
and commutes with all elements of this algebra.
This proves (\ref{P-2902BKR-13a-qCent}) for~$m,$
and completes the proof of the claim.

Now set $p_m = 1 - q_m$ for $m \in \N.$
Part~(\ref{P-2902BKR-13a-PjInAn}) of the conclusion
follows from~(\ref{P-2902BKR-13a-qAm1}).
Part~(\ref{P-2902BKR-13a-pC}) of the conclusion
follows from~(\ref{P-2902BKR-13a-qCent})
and the fact that $C_{m}$ is the subalgebra of~$A$
generated by the subalgebras $\ph_k (\C \oplus D_k)$
for $k = 1, 2, \ldots, m.$
Part~(\ref{P-2902BKR-13a-Dom}) of the conclusion
follows from~(\ref{P-2902BKR-13a-qleq})
and an induction argument using~(\ref{P-2902BKR-13a-qlqm}).
For part~(\ref{P-2902BKR-13a-Sm}) of the conclusion,
we use~(\ref{P-2902BKR-13a-qSubeq}), (\ref{L-2907Pf-Sm}),
and the fact that
$f_1, f_2, \ldots$ are \mops\  in ${\overline{y A y}}.$

We prove (\ref{P-2902BKR-13a-Norm}) of the conclusion.
An induction argument
using (\ref{L-2907Pf-ydef}) and~(\ref{L-2907Pf-Norm})
gives
\[
\| y_m \|
 > 1 - \frac{\ep}{2} - \frac{\ep}{4} - \cdots - \frac{\ep}{2^{m}}.
\]
So $\| y_m \| > 1 - \ep$ for all $m \in \N.$
{}From~(\ref{P-2902BKR-13a-q12}),
we now get $\| (1 - p_m) y (1 - p_m) \| = \| y_{m} \| > 1 - \ep.$

It remains only to prove that
$(A_m)_{m \in \N}$ is strictly increasing.
Since $\ep < 1,$
we have $\| (1 - p_m) y (1 - p_m) \| \neq 0,$
so that $\ph_m (0, 1) \neq 0.$
Thus $\ph_m (\C \oplus D_m)$
contains a summand isomorphic to $M_l$ for some $l \geq 2.$
Therefore~(\ref{Eq:2907Star}) implies that
$A_m$ is a proper subset of $A_{m + 1}.$
\end{proof}

The following result is a weak analog of Corollary~2.10 of~\cite{BKR}.

\begin{cor}\label{C-2907-210}
Let $A$ be a simple separable infinite dimensional unital C*-algebra.
Suppose that $A$ is tracially approximately divisible.
Then for every $\ep > 0,$ every $n, r \in \N$ with $r \geq 2,$
every $a_1, a_2, \ldots, a_n \in A,$
and every $y \in A_{+}$ with $\| y \| = 1,$
there exist
a completely noncommutative finite dimensional C*-algebra~$D$
and a unital \hm\  $\ph \colon \C \oplus D \to A$
such that:
\begin{enumerate}
\item\label{D-2907-210-AC}
$\| \ph (b) a_k - a_k \ph (b) \| < \ep$ for $k = 1, 2, \ldots, n$
and all $b \in \C \oplus D$ with $\| b \| \leq 1.$
\item\label{D-2907-210-Rk}
Every central summand of $D$ has matrix size at least~$r.$
\item\label{D-2907-210-Sm}
$\ph (1, 0)$ is \mvnt\  to a \pj\  in ${\overline{y A y}}.$
\item\label{D-2907-210-Norm}
$\| \ph (0, 1) y \ph (0, 1) \| > 1 - \ep.$
\end{enumerate}
\end{cor}

\begin{proof}
Apply Proposition~\ref{P-2902BKR-13a}
with $y$ and $\ep$ as given,
and use the notation there for the resulting objects.
Choose $m \in \N$ so large that $2^m \geq r$
and that $\dist (a_j, A_m) < \frac{\ep}{2}$
for $j = 1, 2, \ldots, n.$
Set $E = D_1 \otimes D_2 \otimes \cdots \otimes D_m.$
Since $D_k$
is a completely noncommutative finite dimensional C*-algebra
for $k = 1, 2, \ldots, m,$
the matrix size of every simple summand of~$E$ is at least $2^m \geq r.$

Since $\ph_k (\C \oplus D_k) \S A_k,$
it follows from
Proposition \ref{P-2902BKR-13a}(\ref{P-2902BKR-13a-Comm})
that the ranges of $\ph_1, \ph_2, \ldots, \ph_m$ all commute
with each other,
and from
Proposition \ref{P-2902BKR-13a}(\ref{P-2902BKR-13a-pC})
and $p_m$ commutes with the ranges of these \hm s.
Therefore there is a unital \hm\  %
$\ps_0 \colon E \to (1 - p_m) A (1 - p_m)$
such that
\[
\ps_0 (b_1 \otimes b_2 \otimes \cdots \otimes b_m)
 = \ph_1 (0, b_1) \ph_2 (0, b_2) \cdots \ph_m (0, b_m)
\]
for all $b_1 \in D_1, \, b_2 \in D_2, \, \ldots, \, b_m \in D_m.$
Set $D = E / \ker (\ps_0)$
and let $\ps \colon D \to (1 - p_m) A (1 - p_m)$
be the induced \hm.
Clearly
the matrix size of every simple summand of~$D$
is also at least~$r.$

Define $\ph \colon \C \oplus D \to A$
by $\ph (\ld, b) = \ld p_m + \ps_0 (b)$
for $\ld \in \C$ and $b \in D.$
Then $\ph (1, 0)$ is \mvnt\  to a \pj\  in ${\overline{y A y}}$
by Proposition \ref{P-2902BKR-13a}(\ref{P-2902BKR-13a-Sm}),
and $\| \ph (0, 1) y \ph (0, 1) \| > 1 - \ep$
by Proposition \ref{P-2902BKR-13a}(\ref{P-2902BKR-13a-Norm}).
\end{proof}

\begin{rmk}\label{R-2914StComp}
It is almost certainly true that
a tracially approximately divisible \ca\  %
has strict comparison of positive elements
in the sense of Subsection~2.2 of~\cite{Tm}.
If we had defined tracial approximate divisibility
allowing only simple algebras in Definition~\ref{D-2902TrAppDiv},
then in Corollary~\ref{C-2907-210} we could take
$D$ to be simple with arbitrarily large matrix size.
In this case,
it would be immediate that $A$ is tracially $Z$-stable in the sense
of Definition~2.1 of~\cite{HO}.
Theorem~3.3 of~\cite{HO} would then imply
strict comparison of positive elements.

With our definition,
the argument used for the proof of Proposition~2.7 of~\cite{BKR}
allows us to take $D$ in Corollary~\ref{C-2907-210}
to be $M_r \oplus M_{r + 1}.$
So $A$ satisfies a version of tracial $Z$-absorption
using order zero maps from $M_r \oplus M_{r + 1}$ instead of~$M_r.$
The arguments of~\cite{HO}
almost certainly apply under this weaker condition,
and thus should imply that
algebras as in Definition~\ref{D-2902TrAppDiv}
have strict comparison of positive elements.
With this result in hand,
it is easy to use Corollary~\ref{C-2907-210}
with $D = M_s \oplus M_{s + 1},$
with $s$ sufficiently large,
to show that in Corollary~\ref{C-2907-210}
one can in fact take $D = M_r$
and that in Definition~\ref{D-2902TrAppDiv}
one can take $D = M_2.$
\end{rmk}

We do not prove
the conjectures implicit in Remark~\ref{R-2914StComp} here,
because it is already known that tracial rank zero implies the
form of strict comparison that we actually need,
and because, for now, we do not have other interesting examples.

We recall the following definition.

\begin{dfn}\label{D-2910OrdPj}
For a simple unital \ca~$A,$
we say that the
{\emph{order on \pj s over $A$ is determined by traces}}
if whenever $p, q \in \Mi (A)$ are \pj s such that
$\ta (p) < \ta (q)$ for every tracial state $\ta$ on~$A,$
then $p \precsim q.$
\end{dfn}

If the statements in Remark~\ref{R-2914StComp} are correct,
then tracial approximate divisibility of~$A$ implies that
the order on \pj s over~$A$ is determined by traces.
Thus, in the following lemma and its consequences,
this condition can probably be omitted.
Since our main purpose in introducing
tracial approximate divisibility
is to clarify a proof
involving simple \ca s with tracial rank zero,
and since it is already known
that the order on \pj s over such a \ca\   is determined by traces,
we do not prove this here.

\begin{ntn}\label{N-2907Inn}
We let $\Inn (A)$ denote the set of inner automorphisms
$a \mapsto u a u^*$ with $u \in A$ unitary.
We let $\Innb (A)$ be the closure in $\Aut (A)$
of $\Inn (A)$ in the topology of Lemma~\ref{L:AutTop}.
This is the set of approximately inner automorphisms.
\end{ntn}

\begin{lem}\label{InnApprox}
Let $A$ be a simple separable infinite dimensional
unital \ca\  which is tracially approximately divisible
and such that the order on \pj s over~$A$ is determined by traces.
Let $u \in A$ be unitary.
Let $F \S A$ be finite, let $\ep > 0,$ let $n \in \N,$
and let $x \in A$ be a nonzero positive element.
Then, following Notation~\ref{SetNtn},
we have $\Ad (u) \in {\overline{W (F, \ep, n, x) \cap \Inn (A)}}.$
\end{lem}

\begin{proof}
It suffices to prove that for every finite set $T \S A$
and every $\dt > 0,$ there is $\af \in W (F, \ep, n, x) \cap \Inn (A)$
such that $\| \af (a) - u a u^* \| < \dt$ for all $a \in T.$

To simplify the proof, we may assume that $\| a \| \leq 1$
for all $a \in F \cup T.$
Set
\[
S = F \cup T \cup \{ u \}
\andeqn
\ep_0
 = \min \left( \frac{\ep}{2 n + 4}, \, \frac{\dt}{2 n + 6} \right).
\]
Choose $\ep_1 > 0$ such that
$\ep_1 \leq \ep_0$ and whenever $B$ is a unital  \ca\  %
and $a$ and $u$ are elements of $B$
such that $u$ is unitary and $\| a - u \| < (n + 2) \ep_1,$
then there is a unitary $v$ in the unital \ca\  generated by $a$
such that $\| v - u \| < \ep_0.$
Since $A$ has Property~(SP) (by Proposition~\ref{P-2902SP}),
we can use Lemma~1.10 of~\cite{PhT1}
to choose two nonzero orthogonal \pj s $q, r \in {\overline{x A x}}.$
Set $\rh = \inf_{\ta \in T (A)} \ta (q) > 0.$
Choose $m \in \N$ such that
$m > ( 1 + 1 / \rh) (n + 1).$

Use Corollary~\ref{C-2907-210} to find
a \pj~$f \in A,$
a completely noncommutative finite dimensional C*-algebra~$D,$
and an injective unital \hm\  $\ph \colon D \to f A f,$
such that:
\begin{enumerate}
\item\label{D-2905Pf-AC}
$\| \ph (b) a - a \ph (b) \| < \ep_1$ for all $a \in S$
and all $b \in D$ with $\| b \| \leq 1.$
\item\label{D-2905Pf-Sm}
$1 - f$ is \mvnt\  to a \pj\  in $r A r.$
\item\label{D-2905Pf-Big}
There are $N \in \N$ and $d (1), d (2), \ldots, d (N) \in \N$
such that
$D \cong \bigoplus_{l = 1}^N M_{d (l)}$
and $d (l) \geq m$ for $l = 1, 2, \ldots, N.$
\end{enumerate}
For convenience,
we identify $D$ with $\ph (D) \S A,$
and suppress $\ph$ in the notation.
For $l = 1, 2, \ldots, N$ we similarly
identify $M_{d (l)}$ with its image in~$A.$
Let $f_l \in A$ be the identity of $M_{d (l)}.$
Thus $f = \sum_{l = 1}^N f_l.$

For $l = 1, 2, \ldots, N,$
there are $s (l), t (l) \in \Nz$
such that $d (l) = s (l) (n + 1) + t (l)$
and $t (l) \leq n.$
Then
\[
s (l)
 \geq \frac{d (l)}{n + 1} - 1
 \geq \frac{m}{n + 1} - 1
 > \frac{1}{\rh}.
\]
Take ranks with respect to $M_{n + 1}$ and $M_{d (l)}.$
Let $\ps_l \colon M_{n + 1} \to M_{d (l)}$
be any \hm\  (usually not unital)
which sends a rank one \pj\  in $M_{n + 1}$
to a rank $s (l)$ \pj\  in $M_{d (l)}.$
For every $\ta \in T (A),$
we have
\begin{align*}
\ta (f_l - \ps_l (1) )
& = \ta (f_l) \left( \frac{\rank (f_l - \ps_l (1))}{d (l)} \right)
\\
& \leq \frac{\ta (f_l) n}{s (l) (n + 1) + t (l)}
  < \frac{\ta (f_l)}{s (l)}
  < \rh \ta (f_l)
  \leq \ta (q) \ta (f_l).
\end{align*}

Now define a \hm\  $\ps \colon M_{n + 1} \to D$
by $\ps (a) = \sum_{l = 1}^N \ps_l (a).$
For every $\ta \in T (A),$
we get
\[
\ta (f - \ps (1))
 = \sum_{l = 1}^N \ta (f_l - \ps_l (1) )
 < \sum_{l = 1}^N \ta (q) \ta (f_l)
 = \ta (q) \ta (f)
 \leq \ta (q).
\]
Since traces determine the order on \pj s over~$A,$
it follows that $f - \ps (1) \precsim q.$
Combining this with $1 - f \precsim r,$
we find that $1 - \ps (1)$ is \mvnt\  to
a \pj\  in ${\overline{x A x}}.$
Moreover,
$\| \ps (b) a - a \ps (b) \| < \ep_1$ for all $a \in S$
and all $b \in M_{n + 1}$ with $\| b \| \leq 1.$

Let $( g_{j, k} )_{0 \leq j, k \leq n}$
be a system of matrix units for $M_{n + 1}.$
Define \pj s $e_j \in A$ by
$e_j = \ps (g_{j, j})$ for $j = 0, 1, \ldots, n.$

We claim that for every $a \in S$ there is $c \in A$
such that $\| c - a \| < (n + 2) \ep_1$
and $c$ commutes with every element of
$\C (1 - \ps (1)) + \ps (M_{n + 1}).$
Given~$a,$
set
\[
c = (1 - \ps (1)) a (1 - \ps (1))
     + \sum_{j = 0}^{n} \ps (g_{j, 0}) a \ps (g_{0, j}).
\]
The commutation relation is easily checked.
Also,
\begin{align*}
\| c - a \|
& = \left\| c - \left( a [1 - \ps (1)]
            + \sum_{j = 0}^{n} a \ps (g_{j, j}) \right) \right\|
      \\
& \leq \big\| [1 - \ps (1)] a [1 - \ps (1)] - a [1 - \ps (1)] \big\|
       + \sum_{j = 0}^{n}
           \| \ps (g_{j, 0}) a \ps (g_{0, j}) - a \ps (g_{j, j}) \|
      \\
& \leq \big\| [1 - \ps (1)] a - a [1 - \ps (1)] \big\|
       + \sum_{j = 0}^{n}
           \| \ps (g_{j, 0}) a - a \ps (g_{j, 0}) \|
      \\
& < (n + 2) \ep_1.
\end{align*}
This proves the claim.

We next claim that there is a unitary $v \in A$
such that $v$ commutes with every element of
$\C (1 - \ps (1)) + \ps (M_{n + 1})$
and $\| v - u \| < \ep_0.$
Let $c$ be as in the previous claim,
with $a = u.$
The choice of $\ep_1$ implies that there is a unitary $v$ in the
\ca\  generated by~$c$ such that $\| v - u \| < \ep_0.$
Clearly $v$ commutes with every element of
$\C (1 - \ps (1)) + \ps (M_{n + 1}).$
The claim is proved.

Define a unitary $w \in A$ by
\[
w = 1 - \ps (1) + \ps \left( g_{0, n}
     + \sum_{j = 0}^{n - 1} g_{j + 1, \, j} \right).
\]
Then define $\af \in \Aut (A)$ by $\af = \Ad (w v).$
We claim that $\| \af (a) - u a u^* \| < \dt$ for all $a \in T$
and that $\af \in W (F, \ep, n, x).$
This will complete the proof of the lemma.
By construction,
we have
\begin{equation}\label{Eq:vwComm}
w v = v w.
\end{equation}

We prove the first part.
Let $a \in T.$
Since $T \S S,$
there is $c \in A$
such that $\| c - a \| < (n + 2) \ep_1$
and $c$ commutes with every element of
$\C (1 - \ps (1)) + \ps (M_{n + 1}).$
In particular, $c$ commutes with~$w.$
Using this and~(\ref{Eq:vwComm}) at the fourth step,
we therefore get
\begin{align*}
\| \af (a) - u a u^* \|
& = \| w v a v^* w^* - u a u^* \|
  \leq 2 \| v - u \| + \| w v a v^* w^* - v a v^* \|
    \\
& \leq 2 \| v - u \| + 2 \| a - c \|
         + \| w v c v^* w^* - v c v^* \|
  \\
& = 2 \| v - u \| + 2 \| a - c \|
  < 2 \ep_0 + 2 (n + 2) \ep_1
\leq \dt.
\end{align*}

We now prove that $\af \in W (F, \ep, n, x).$
We start with condition~(\ref{2906SetNtn-1}) of
Notation~\ref{SetNtn}.
For $j = 0, 1, \ldots, n - 1,$
we have $v e_j v^* = e_j$
since $e_j \in \ps (M_{n + 1}).$
Also,
it is easy to check that $w e_j w^* = e_{j + 1}.$
Therefore $\af (e_j) = e_{j + 1}.$

For condition~(\ref{2906SetNtn-2}),
let $a \in F.$
Since $T \S S,$
there is $c \in A$
such that $\| c - a \| < (n + 2) \ep_1$
and $c$ commutes with every element of
$\C (1 - \ps (1)) + \ps (M_{n + 1}).$
In particular, $c e_j = e_j c$
for $j = 0, 1, \ldots, n.$
Therefore
\[
\| e_j a - a e_j \|
  \leq 2 \| a - c \| + \| e_j c - c e_j \|
  = 2 \| a - c \|
  < 2 (n + 2) \ep_1
  \leq \ep.
\]

For condition~(\ref{2906SetNtn-3}),
we have $\sum_{j = 0}^{n} e_j = \ps (1)$
and we already saw that $1 - \ps (1)$
is \mvnt\  to a
\pj\  in ${\overline{x A x}}.$
\end{proof}

\begin{thm}\label{InnTAD}
Let $A$ be a simple separable infinite dimensional
unital \ca\  which is tracially approximately divisible
and such that the order on \pj s over~$A$ is determined by traces.
Then there is a dense $G_{\dt}$-set $G \S \Innb (A)$
such that every $\af \in G$ has the \trp.
\end{thm}

\begin{proof}
The set $\Innb (A)$ is closed in $\Aut (A),$ hence complete.
By Lemma~\ref{Intersect},
there is a countable intersection $G$ of
sets of the form $W (F, \ep, n, x) \cap \Innb (A)$
such that every $\af \in G$ has the \trp.
By Lemma~\ref{Open}, each of the sets $W (F, \ep, n, x)$
is open in $\Innb (A),$
and by Lemma~\ref{InnApprox},
each of the sets $W (F, \ep, n, x)$
is dense in $\Innb (A).$
\end{proof}

\begin{prp}\label{P-2907TRZTAD}
Let $A$ be a simple separable infinite dimensional
unital \ca\  with tracial rank zero.
Then $A$ is tracially approximately divisible.
\end{prp}

\begin{proof}
We verify the condition of Lemma~\ref{L-2907FinNoy}.
So let $\ep > 0,$ let $n \in \N,$
let $a_1, a_2, \ldots, a_n \in A,$
and let $y \in A_{+} \SM \{ 0 \}.$
Since tracial rank zero implies real rank zero for simple \ca s,
we can use Lemma~1.10 of~\cite{PhT1}
to choose two nonzero orthogonal \pj s
$q_1, q_2 \in {\overline{y A y}}.$
Set $\rh = \inf_{\ta \in T (A)} \ta (q_1).$
Then $\rh > 0.$
Choose $m \in \N$ such that $m > \max ( 2, \, 1 / \rh).$
Then $M_m$
is a completely noncommutative finite dimensional C*-algebra.

Since $A$ has tracial rank zero,
by Theorem~6.13 of~\cite{LnTTR}
there are a \pj\  $f \in A$ and a unital finite dimensional
subalgebra $E \S f A f$ such that:
\begin{enumerate}
\item\label{P-2907TRZTAD-1}
$\| f a_j - a_j f \| < \frac{\ep}{6}$ for $j = 1, 2, \ldots, n.$
\item\label{P-2907TRZTAD-2}
For $j = 1, 2, \ldots, n$ there exists $b_j \in E$ such that
$\| f a_j f - b_j \| < \frac{\ep}{6}.$
\item\label{P-2907TRZTAD-3}
$1 - f \precsim q_2.$
\end{enumerate}
We may write $E \cong \bigoplus_{l = 1}^N M_{d (l)}$
for suitable $N \in \N$ and $d (1), \ldots, d (N) \in \N.$
Choose matrix units $f_{r, s}^{(l)} \in M_{d (l)}$
for $l = 1, 2, \ldots, N$ and $r, s = 1, 2, \ldots, d (l).$
Also let $f_l = \sum_{r = 1}^{d (l)} f_{r, r}^{(l)}$
be the identity of $M_{d (l)} \S E.$

Let
$( e_{j, k} )_{1 \leq j, k \leq m}$
be the standard system of matrix units in~$M_m.$
For $l = 1, 2, \ldots, N,$
apply Lemma~2.3 of~\cite{OP1} to $f_{1, 1}^{(l)} A f_{1, 1}^{(l)}.$
We obtain \mvnt\  \mops\  %
$p_{j}^{(l)} \leq f_{1, 1}^{(l)}$ for $j = 1, 2, \ldots, m,$
such that
\begin{equation}\label{Eq:2909GetP}
f_{1, 1}^{(l)} - \sum_{j = 1}^{m} p_{j}^{(l)}
               \precsim p_{1}^{(l)}.
\end{equation}
Extend $p_1^{(l)}, p_2^{(l)}, \ldots, p_m^{(l)}$
to a system of matrix units
$\big( g_{j, k}^{(l)} \big)_{1 \leq j, k \leq m}$
of type~$M_m$ in~$f_{1, 1}^{(l)} A f_{1, 1}^{(l)}$
such that $g_{j, j}^{(l)} = p_j^{(l)}$ for $j = 1, 2, \ldots, m.$
Then there is a \hm\  %
$\ps_l \colon M_m \to f_{1, 1}^{(l)} A f_{1, 1}^{(l)}$
such that $\ps_l (e_{j, k}) = g_{j, k}^{(l)}$
for $j, k = 1, 2, \ldots, d (l).$
Define $\ph_l \colon M_m \to f_l A f_l$
by
\[
\ph_l (x)
 = \sum_{r = 1}^{d (l)} f_{r, 1}^{(l)} \ps_l (x) f_{1, r}^{(l)}.
\]
For every $x \in M_m,$
the element $\ps_l (x)$ commutes
with every element of $M_{d (l)} \S E \S A.$
Moreover,
for every $\ta \in T (A)$ we have,
using~(\ref{Eq:2909GetP}) at the second step,
\[
\ta \big( f_{1, 1}^{(l)} - \ps_l (1) \big)
  \leq \ta \big( p_{1}^{(l)} \big)
  \leq \frac{ \ta \big( f_{1, 1}^{(l)} \big)}{m}.
\]
Therefore
$\ta ( f_l - \ph_l (1) ) \leq \ta (f_l) / m.$

Define a \hm\  $\ph \colon M_m \to A$
by $\ph (x) = \sum_{l = 1}^N \ph_l (x).$
Then $\ph (x)$ commutes
with every element of $E$ for all $x \in M_m.$
Moreover,
for every $\ta \in T (A),$
\[
\ta (f - \ph (1))
 = \sum_{l = 1}^N \ta ( f_l - \ph_l (1) )
 \leq \frac{1}{m} \sum_{l = 1}^N \ta (f_l)
 \leq \frac{1}{m}
 < \rh
 \leq \ta (q_2).
\]
Since traces determine the order on \pj s over~$A$
(by Corollary~5.7 and Theorem~6.8 of~\cite{LnTTR}),
it follows that $f - \ph (1) \precsim q_2.$
Since $1 - f \precsim q_1,$
it follows that $1 - \ph (1) \precsim q_1 + q_2 \in {\overline{y A y}}.$
Setting $e = \ph (1),$
we have proved condition~(\ref{D-2907FinNoy-Sm})
of Lemma~\ref{L-2907FinNoy}.

It remains to prove condition~(\ref{D-2907FinNoy-AC})
of Lemma~\ref{L-2907FinNoy}.
For $j = 1, 2, \ldots, n,$
let $b_j \in E$ be as in~(\ref{P-2907TRZTAD-2}).
Set $c_j = (1 - f) a_j (1 - f) + b_j.$
Then
\begin{align*}
\| c_j - a_j \|
& \leq \| (1 - f) a_j f \| + \| f a_j (1 - f) \| + \| b_j - f a_j f \|
   \\
& \leq 2 \| f a_j - a_j f \| + \| b_j - f a_j f \|
  < \frac{2 \ep}{6} + \frac{\ep}{6}
  = \frac{\ep}{2}.
\end{align*}
Now let $x \in M_m$ satisfy $\| x \| \leq 1.$
Then $\ph (x) \in f A f$ and $\ph (x)$ commutes with~$b_j.$
So $\ph (x)$ commutes with~$c_j.$
Therefore
\[
\| \ph (x) a_j - a_j \ph (x) \|
  \leq 2 \| c_j - a_j \|
  < \ep.
\]
This completes the proof.
\end{proof}

\begin{thm}\label{InnMain}
Let $A$ be a simple separable infinite dimensional
unital \ca\  with tracial rank zero.
Then there is a dense $G_{\dt}$-set $G \S \Innb (A)$
such that every $\af \in G$ has the \trp.
\end{thm}

\begin{proof}
The algebra $A$ is tracially approximately divisible
by Proposition~\ref{P-2907TRZTAD}.
It follows from Corollary~5.7 and Theorems~5.8 and~6.8 of~\cite{LnTTR}
that traces determine the order on \pj s over~$A.$
So Theorem~\ref{InnTAD} applies.
\end{proof}

\section{Purely infinite simple C*-algebras}\label{Sec:PISCA}

\indent
In this section,
we consider the Rokhlin property
rather than the tracial Rokhlin property.
We prove that it is generic for actions of $\Z$ or $\Z^d$
on a unital \ca\  which tensorially absorbs~${\mathcal{O}}_{\I}$
or a UHF algebra of infinite type.
Essentially the same proof shows that if $G$ is a finite group
with $r$~elements,
then the Rokhlin property is generic
for actions of $G$
on a unital \ca\  which tensorially absorbs
the $r^{\I}$~UHF algebra.

We give the proof in detail for actions of $\Z^d$
on ${\mathcal{O}}_{\I}$-absorbing \ca s for $d \geq 2.$
Because of technical differences in the definitions,
a combinatorial argument (see Lemma~\ref{R-2909RZdRZ})
is needed to show that this definition in the case $d = 1$
implies the usual definition for actions of~$\Z.$
(In the simple case,
the reverse implication is unknown.)

We recall the definition of the Rokhlin property for actions of~$\Z^d$
from~\cite{Nk1}.
However,
we split it in several parts for convenient reference in our arguments.

\begin{ntn}\label{N-2908-ZdN}
Let $d \in \N.$
We say that
$m \in \Z^d$
is {\emph{strictly positive}}
if $m_k > 0$ for $k = 1, 2, \ldots, d.$
We then denote by $m \Z^d$ the set
$\prod_{k = 1}^d m_k \Z \S \Z^d.$
As convenient,
we identify $\Z^d / m \Z^d$
with $\prod_{k = 1}^d m_k \Z / m_k \Z$
or as
\[
\Z^d / m \Z^d
 = \big\{ n \in \Z^d \colon
        {\mbox{$0 \leq n_k < m_k$ for $k = 1, 2, \ldots, d$}} \big\}.
\]
We further let $\dt_1, \dt_2, \ldots, \dt_d \in \Z^d$
be the standard basis vectors for $\R^d.$
For $r \in \Z^d$ and $l \in \Z^d / m \Z^d,$
we interpret $r + l$ as the element of $\Z^d / m \Z^d$
gotten as the sum of $l$ and the image of $r$ in $\Z^d / m \Z^d.$
\end{ntn}

\begin{dfn}\label{D-2908-ZdOneRT}
Let $A$ be a unital \ca,
let $d \in \N,$
and let $\af \colon \Z^d \to \Aut (A).$
Let $F \S A$ be finite, let $\ep > 0,$
and let and let $m \in \Z^d$
be strictly positive.
Then an {\emph{$(m, F, \ep)$-Rokhlin tower for~$\af$}}
is a family $(e_l)_{l \in \Z^d / m \Z^d}$
of \mops\  such that:
\begin{enumerate}
\item\label{D-2908-ZdOneRT-Cm}
$\| e_l a - a e_l \| < \ep$ for all $a \in F$ and $l \in \Z^d / m \Z^d.$
\item\label{D-2908-ZdOneRT-Tr}
$\| \af_{\dt_k} (e_l) - e_{l + \dt_k} \| < \ep$
for all $l \in \Z^d / m \Z^d$ and for $k = 1, 2, \ldots, d.$
\end{enumerate}
We say it is an {\emph{exact $(m, F)$-Rokhlin tower for~$\af$}}
if the norms in (\ref{D-2908-ZdOneRT-Cm}) and~(\ref{D-2908-ZdOneRT-Tr})
are all zero.
Its {\emph{support}} is defined to be $\sum_{l \in \Z^d / m \Z^d} e_l.$
\end{dfn}

\begin{dfn}\label{D-2908-ZdRTS}
Let $A$ be a unital \ca,
let $d \in \N,$
and let $\af \colon \Z^d \to \Aut (A).$
Let $F \S A$ be finite and let $\ep > 0.$
An {\emph{$(F, \ep)$-system of Rokhlin towers for~$\af$}}
consists of $s \in \N,$
an $s$-tuple $m = \left( m^{(1)}, m^{(2)}, \ldots, m^{(s)} \right)$
of strictly positive elements of~$\Z^d,$
and $(m^{(j)}, F, \ep)$-Rokhlin towers for~$\af$
for $j = 1, 2, \ldots, s$
whose supports $f_j$ are orthogonal
and satisfy $\sum_{j = 1}^s f_j = 1.$
We call such a system
an {\emph{exact $F$-system of Rokhlin towers for~$\af$}}
if the constituent Rokhlin towers are all
exact $(m^{(j)}, F)$-Rokhlin towers for~$\af.$
The {\emph{size}} of such a system is~$m.$
Its {\emph{minimum height}},
or the minimum height
of $\left( m^{(1)}, m^{(2)}, \ldots, m^{(s)} \right),$
is
\[
\min \big( \big\{ m^{(j)}_k \colon
   {\mbox{$1 \leq j \leq s$ and $1 \leq k \leq d$}} \big\} \big).
\]
\end{dfn}

\begin{dfn}[Definition~1 of~\cite{Nk1}]\label{D-2908-Zd-RP}
Let $A$ be a separable unital \ca,
let $d \in \N,$
and let $\af \colon \Z^d \to \Aut (A).$
Then $\af$ has the
{\emph{Rokhlin property}}
if for every $N \in \N$
there are $s \in \N$
and an $s$-tuple $m = \left( m^{(1)}, m^{(2)}, \ldots, m^{(s)} \right)$
of strictly positive elements of~$\Z^d$
with minimum height at least~$N$
such that for every finite subset $F \S A$ and every $\ep > 0$
there is an $(F, \ep)$-system of Rokhlin towers for~$\af$
of size~$m.$
\end{dfn}

For comparison,
we recall what is now the usual definition
of the Rokhlin property
for a single automorphism of a \ca.

\begin{dfn}[Definition~2.5 of~\cite{Iz0}]\label{D-IzumiRP}
Let $A$ be a separable unital \ca,
and let $\af \in \Aut (A).$
We say that $\af$ has the
{\emph{Rokhlin property}}
if for every $\ep > 0,$ every finite subset $F \S A,$
and every $n \in \N,$
there are \mops\  %
$e_0, \, e_1, \, \ldots, \, e_{n - 1},
       \, f_0, \, f_1, \, \ldots, \, f_n \in A$
such that:
\begin{enumerate}
\item\label{Pf-2909RZdRZ-1}
$\| \af (e_l) - e_{l + 1} \| < \ep$ for $l = 0, 1, \ldots, n - 2$
and $\| \af (f_l) - f_{l + 1} \| < \ep$ for $l = 0, 1, \ldots, n - 1.$
\item\label{Pf-2909RZdRZ-2}
$\| e_l a - a e_l \| < \ep$
for $l = 0, 1, \ldots, n - 1$ and all $a \in F,$
and
$\| f_l a - a f_l \| < \ep$ for $l = 0, 1, \ldots, n$ and all $a \in F.$
\item\label{Pf-2909RZdRZ-3}
$\sum_{l = 0}^{n - 1} e_l + \sum_{l = 0}^{n} f_l = 1.$
\end{enumerate}
\end{dfn}

Informally,
the Rokhlin property of Definition~\ref{D-2908-Zd-RP}
for actions of~$\Z$
requires that the top \pj\  of each tower
be sent by the automorphism
to a \pj\  close to the base \pj\  of the same tower.
However,
the Rokhlin property of Definition~2.5 of~\cite{Iz0},
while more restrictive about the sizes of towers,
only requires that the sum of the top \pj s of the towers
be sent by the automorphism to a \pj\  close to the
sum of the base \pj s of the towers.
We know of no proof that this version implies
the Rokhlin property of Definition~\ref{D-2908-Zd-RP},
but, at least for simple \ca s,
we know of no counterexamples either.

It is known that the Rokhlin property of
Definition~\ref{D-2908-Zd-RP} for $d = 1$
implies that of Definition~\ref{D-IzumiRP},
but it seems not to have been
explicitly written down anywhere.
So we give a proof here.
It uses a tower partitioning argument.
Several somewhat related proofs appear in Section~1 of~\cite{OP1}.
(Strictly speaking,
we can avoid this result,
since the proofs we give,
restricted to $d = 1,$
already give towers of the form in Definition~\ref{D-2908-Zd-RP}.)

\begin{lem}\label{R-2909RZdRZ}
For $d = 1,$
the Rokhlin property of Definition~\ref{D-2908-Zd-RP}
implies the Rokhlin property of Definition~\ref{D-IzumiRP}.
\end{lem}

\begin{proof}
Write $\af$ for the automorphism generating the action
as well as the action.
Let $\ep > 0,$ let $F \S A$ be finite, and let $n \in \N.$

Choose $t \in \N$ and
a $t$-tuple $m = \left( m^{(1)}, m^{(2)}, \ldots, m^{(t)} \right)$
of strictly positive elements of~$\Z$
with minimum height at least~$n^2$
and such that for every finite subset $S \S A$ and every $\dt > 0$
there is an $(S, \dt)$-system of Rokhlin towers for~$\af$
of size~$m.$
For $j = 1, 2, \ldots, t,$
since $m^{(j)} \geq n^2,$
there are $r_j, s_j \in \Nz$
such that $m^{(j)} = r_j n + s_j (n + 1).$
Therefore
there are intervals $I_{j, k} \S \Nz$ for $k = 1, 2, \ldots, r_j$
with exactly $n$ elements
and intervals $J_{j, k} \S \Nz$ for $k = 1, 2, \ldots, s_j$
with exactly $n + 1$ elements
such that
\[
\big\{ 0, 1, \ldots, m^{(j)} - 1 \big\}
 = \coprod_{k = 1}^{r_j} I_{j, k} \amalg \coprod_{k = 1}^{s_j} J_{j, k}.
\]
Thus, for $k = 1, 2, \ldots, r_j$ there is $\mu_{j, k} \in \Nz$
such that
\[
I_{j, k} = \big\{ \mu_{j, k}, \, \mu_{j, k} + 1, \, \ldots,
                   \mu_{j, k} + n - 1 \big\}
\]
and for $k = 1, 2, \ldots, s_j$ there is $\nu_{j, k} \in \Nz$
such that
\[
J_{j, k} = \big\{ \nu_{j, k}, \, \nu_{j, k} + 1, \, \ldots,
                   \nu_{j, k} + n \big\}.
\]

Set
\[
r = \sum_{j = 1}^t r_j,
\,\,\,\,\,\,
s = \sum_{j = 1}^t s_j,
\andeqn
\ep_0 = \frac{\ep}{\max (r, s)}.
\]
Choose an $(F, \ep_0)$-system of Rokhlin towers for~$\af$
of size~$m,$
say
\[
R = \left( \big( p_0^{(j)}, \, p_1^{(j)}, \, \ldots,
  \, p_{m^{(j)} - 1}^{(j)} \big) \right)_{j = 1}^t.
\]
Define \pj s $e_l$
for $l = 0, 1, \ldots, n - 1$
and $f_l$ for $l = 0, 1, \ldots, n$ by
\[
e_l = \sum_{j = 1}^t \sum_{k = 1}^{r_j} p_{\mu_{j, k} + l}^{(j)}
\andeqn
f_l = \sum_{j = 1}^t \sum_{k = 1}^{s_j} p_{\nu_{j, k} + l}^{(j)}.
\]
One easily checks that each \pj\  in the system~$R$
has been used exactly once,
which implies~(\ref{Pf-2909RZdRZ-3}) above.
For~(\ref{Pf-2909RZdRZ-1}),
for $l = 0, 1, \ldots, n - 2$
we estimate
\[
\| \af (e_l) - e_{l + 1} \|
 \leq \sum_{j = 1}^t \sum_{k = 1}^{r_j}
    \big\| \af \big( p_{\mu_{j, k} + l}^{(j)} \big)
            - p_{\mu_{j, k} + l + 1}^{(j)} \big\|
 < r \ep_0
 \leq \ep.
\]
A similar calculation gives
$\| \af (f_l) - f_{l + 1} \| < s \ep_0 \leq \ep.$
For~(\ref{Pf-2909RZdRZ-2}),
for $l = 0, 1, \ldots, n - 1$ and $a \in F$
we estimate
\[
\| e_l a - a e_l \|
 \leq \sum_{j = 1}^t \sum_{k = 1}^{s_j}
    \big\| p_{\nu_{j, k} + l}^{(j)} a
              - a p_{\nu_{j, k} + l}^{(j)} \big\|
 < r \ep_0
 \leq \ep
\]
and similarly
for $l = 0, 1, \ldots, n$ and $a \in F$
we get
$\| f_l a - a f_l \| < s \ep_0 \leq \ep.$
This completes the proof.
\end{proof}

\begin{ntn}\label{N-2910Act}
Let $A$ be a separable unital \ca,
and let $d \in \N.$
We denote by ${\mathrm{Act}}_d (A)$
the set of actions of $\Z^d$ on~$A.$
We identify ${\mathrm{Act}}_d (A)$ with a subset
of $\Aut (A)^d$
via the map
$\af \mapsto
 \big( \af_{\dt_1}, \, \af_{\dt_2}, \ldots, \, \af_{\dt_d} \big).$
For $\af, \bt \in {\mathrm{Act}}_d (A)$
and any enumeration $S = (a_1, a_2, \ldots )$ of a countable dense
subset of~$A,$
we let $\rh_S$ be as in Notation~\ref{N-2910ZMetric},
and define
\[
\rh_{d, S} (\af, \bt)
 = \max \big( \rh_S ( \af_{\dt_1}, \bt_{\dt_1} ), \,
              \rh_S ( \af_{\dt_2}, \bt_{\dt_2} ), \,
              \ldots, \,
              \rh_S ( \af_{\dt_d}, \bt_{\dt_d} ) \big).
\]
\end{ntn}

\begin{lem}\label{L-2910PIisMet}
For any $S$ as in Notation~\ref{N-2910Act},
the function $\rh_{d, S}$
is a complete metric on ${\mathrm{Act}}_d (A)$
which induces the restriction to ${\mathrm{Act}}_d (A)$
of the product topology on $\Aut (A)^d.$
\end{lem}

\begin{proof}
It is easy to check that ${\mathrm{Act}}_d (A)$
is a closed subset of $\Aut (A)^d.$
The rest is immediate from Lemma~\ref{L:AutTop}.
\end{proof}

\begin{ntn}\label{N-2910PISetNtn}
Let $A$ be a unital \ca,
and let $d \in \N.$
For a finite set $F \S A,$ for $\ep > 0,$ for $s \in \N,$
and for
an $s$-tuple $m = \left( m^{(1)}, m^{(2)}, \ldots, m^{(s)} \right)$
of strictly positive elements of~$\Z^d,$
we define $W_d (F, \ep, s, m) \S {\mathrm{Act}}_d (A)$
to be the set of all
$\af \in {\mathrm{Act}}_d (A)$ such that
there is an $(F, \ep)$-system of Rokhlin towers for~$\af$
of size~$m.$
\end{ntn}

\begin{lem}\label{L-2910PIOpen}
Let $A$ be a unital \ca,
and let $d \in \N.$
For every finite set $F \S A,$ every $\ep > 0,$ every $s \in \N,$
and every $s$-tuple
$m = \left( m^{(1)}, m^{(2)}, \ldots, m^{(s)} \right)$
of strictly positive elements of~$\Z^d,$
the set $W_d (F, \ep, s, m)$ is open in ${\mathrm{Act}}_d (A).$
\end{lem}

\begin{proof}
The proof is similar to that of Lemma~\ref{Open}.
\end{proof}

\begin{lem}\label{L-2910PIIntersect}
Let $A$ be a separable unital \ca,
and let $d \in \N.$
Let $(s_n)_{n \in \N}$ be a sequence in~$\N,$
and for $n \in \N$ let $m (n)$ be an $s_n$-tuple
of strictly positive elements of $\Z^d$ with minimum height $N_n.$
Suppose that $\limi{n} N_n = \I.$
Then the set of $\af \in {\mathrm{Act}}_d (A)$
which have the Rokhlin property
contains the
countable intersection of sets of the form
$W_d (F, \ep, s_n, m (n))$
for finite sets $F \S A,$
numbers $\ep > 0,$
and integers $n \in \N.$
\end{lem}

\begin{proof}
Choose a countable dense subset $S \S A,$
and let ${\mathcal{F}}$ be the set of all finite subsets of~$S.$
Then clearly
every action
\[
\af \in \bigcap_{F \in {\mathcal{F}}}
            \bigcap_{r = 1}^{\I} \bigcap_{n = 1}^{\I}
            W_d \left( F, \tfrac{1}{r}, s_n, m (n) \right)
\]
has the Rokhlin property.
\end{proof}

\begin{lem}\label{L-2910PIApprox}
Let $A$ be a separable unital \ca\  %
such that ${\mathcal{O}}_{\I} \otimes A \cong A.$
Let $d \in \N,$ let $F \S A$ be finite, let $\ep > 0,$
and let $n \in \N.$
Let $m$ be the $2^d$-tuple consisting of
all elements of
$\{ n, \, n + 1 \}^d,$
arranged in any order.
Then $W_d (F, \ep, 2^d, m)$ is dense in ${\mathrm{Act}}_d (A).$
\end{lem}

\begin{proof}
It suffices to prove that for
every $\af \in {\mathrm{Act}}_d (A),$
every finite set $T \S A,$
and every $\dt > 0,$ there is $\bt \in W_d (F, \ep, 2^d, m)$
such that
$\| \bt_{\dt_k} (a) - \af_{\dt_k} (a) \| < \dt$
for $k = 1, 2, \ldots, d$ and for all $a \in T.$

Set $S = F \cup T \cup \bigcup_{k = 1}^d \af_{\dt_k} (T).$

By Theorem~3.15 of~\cite{KP} and Theorem~3.3 of~\cite{LP},
there is an isomorphism
$\ph_0 \colon {\mathcal{O}}_{\I} \otimes {\mathcal{O}}_{\I}
   \to {\mathcal{O}}_{\I}$
such that the map $a \mapsto \ph_0 (1 \otimes a)$ is
approximately unitarily equivalent to $\id_{{\mathcal{O}}_{\I}}.$
Replacing the first tensor factor ${\mathcal{O}}_{\I}$
in the domain
by the $d$-fold tensor product $({\mathcal{O}}_{\I})^{\otimes d}$
(which is isomorphic to ${\mathcal{O}}_{\I}$),
and using ${\mathcal{O}}_{\I} \otimes A \cong A,$
we get an isomorphism
$\ph \colon ({\mathcal{O}}_{\I})^{\otimes d} \otimes A \to A$
such that the map $a \mapsto \ph (1 \otimes a)$ is
approximately unitarily equivalent to $\id_{A}.$

It is easy to find an injective unital \hm\  %
$\rh_0 \colon M_n \oplus M_{n + 1} \to {\mathcal{O}}_{\I}.$
(For rank one \pj s $e \in M_n$ and $f \in M_{n + 1},$
it will send $(e, 0)$
to a projection $p \in {\mathcal{O}}_{\I}$ such that $[p] = - [1]$
in $K_0 ( {\mathcal{O}}_{\I} )$
and $(0, f)$ to a projection $q \in {\mathcal{O}}_{\I}$
such that $[q] = [1]$
in $K_0 ( {\mathcal{O}}_{\I} ).$)
Using cyclic permutation matrices in $M_n$ and $M_{n + 1},$
find \pj s $f_0^{(0)}, f_1^{(0)}, \ldots, f_{n - 1}^{(0)} \in M_n$
with $\sum_{j = 0}^{n - 1} f_j^{(0)} = 1,$
\pj s $f_0^{(1)}, f_1^{(1)}, \ldots, f_{n}^{(1)} \in M_{n + 1}$
with $\sum_{j = 0}^{n} f_j^{(1)} = 1,$
and unitaries $v^{(0)} \in M_n$ and $v^{(1)} \in M_{n + 1},$
such that
$v^{(0)} f_j^{(0)} \big(v^{(0)})^* = f_{j + 1}^{(0)}$
for $j = 0, 1, \ldots, n - 1$
(with indices taken mod~$n$)
and $v^{(1)} f_j^{(1)} \big(v^{(1)})^* = f_{j + 1}^{(1)}$
for $j = 0, 1, \ldots, n$
(with indices taken mod~$n + 1$).
Define
\[
\rh_1 = (\rh_0)^{\otimes d} \colon
   (M_n \oplus M_{n + 1})^{\otimes d}
    \to ({\mathcal{O}}_{\I})^{\otimes d}
\andeqn
\rh = \ph \circ \rh_1 \colon
   (M_n \oplus M_{n + 1})^{\otimes d}
    \to A.
\]

For $k = 1, 2, \ldots, d,$
define a unitary $w_k \in (M_n \oplus M_{n + 1})^{\otimes d}$
by
$w_k = z_1 \otimes z_2 \otimes \cdots \otimes z_d$
with $z_k = \big( v^{(0)}, v^{(1)} \big)$
and $z_j = 1$ for $j \neq k.$
Let
$\gm \colon
 \Z^d \to \Aut \big( ({\mathcal{O}}_{\I})^{\otimes d} \big)$
be the action such that
$\gm_{\dt_k} = \Ad ( \rh_1 (w_k))$ for $k = 1, 2, \ldots, d.$
For $\mu \in \{ n, \, n + 1 \}^d$ and $k = 1, 2, \ldots, d,$
set $\ep_k = \mu_k - n \in \{ 0, 1 \}.$
Define an exact $(\mu, \{ 1 \})$-Rokhlin tower $R_{\mu}$
for~$\gm$
by applying $\rh_1$ to the tensor product of the
towers $\big( f_j^{(\ep_k)} \big)_{j = 0}^{\mu_k - 1}$
in $M_{\mu_k} \S M_n \oplus M_{n + 1}.$
That is,
for $l \in \Z^d$
with $0 \leq l_k < \mu_k$ for $k = 1, 2, \ldots, d,$
set
\[
e_l = f_{l_1}^{(\mu_1 - n)} \otimes f_{l_2}^{(\mu_2 - n)}
       \otimes \cdots \otimes f_{l_d}^{(\mu_d - n)},
\]
regard $e_l$ as an element of $(M_n \oplus M_{n + 1})^{\otimes d},$
and set $R_{\mu} = ( \rh_1 (e_l) )_{l \in \Z^d / \mu \Z^d}.$
The towers $R_{\mu}$ together
form an exact $(\{ 1 \})$-system $R$ of Rokhlin towers for~$\gm$
of size~$m.$

Set $\ep_0 = \min \left( \frac{\ep}{2}, \frac{\dt}{2} \right).$
Choose $u \in A$ such that $\| u \ph (1 \otimes a) u^* - a \| < \ep_0$
for all $a \in S.$
Set $\ps = \Ad (u) \circ \ph,$
and for $l \in \Z^d$ set
\[
\bt_{l}
 = \ps \circ \big( \gm_{l} \otimes \af_{l} \big) \circ \ps^{-1}.
\]

We claim that $\bt \in W_d (F, \ep, 2^d, m).$
Tensor all \pj s in all towers in~$R$ with $1_A,$
to get a system
$R \otimes 1
 = \big( R_{\mu} \otimes 1 \big)_{\mu \in \{ n, \, n + 1 \}^d}$
of Rokhlin towers
for the action $\gm \otimes \af$ of size~$m$
which is an exact $G$-system for any
finite set
$G \S 1 \otimes A \S ({\mathcal{O}}_{\I})^{\otimes d} \otimes A.$
Set
\[
1 \otimes S
 = \{ 1 \otimes a \colon a \in S \}
 \S ({\mathcal{O}}_{\I})^{\otimes d} \otimes A.
\]
Then $R \otimes 1$
is an exact $1 \otimes S$-system of Rokhlin towers
for $\gm \otimes \af,$
and $u \ph (R \otimes 1) u^*$
(with the obvious meaning)
is an exact $u \ph (1 \otimes S) u^*$-system
of Rokhlin towers for~$\bt.$
Let $\mu \in \{ n, \, n + 1 \}^d.$
Write $u \ph (R_{\mu} \otimes 1) u^* = (q_l)_{l \in \Z^d / \mu \Z^d}.$
Let $l \in \Z^d / \mu \Z^d.$
We must show that
$\| \bt_{\dt_k} (q_l) - q_{l + \dt_k} \| < \ep$
for $k = 1, 2, \ldots, d$
(recall that addition is mod $\mu \Z^d$)
and that $\| a q_l - q_l a \| < \ep$ for all $a \in F.$
For the first,
we have $\bt_{\dt_k} (q_l) = q_{l + \dt_k}$ by construction.
For the second,
by construction $u \ph (1 \otimes a) u^*$ commutes exactly with $q_l.$
The choice of $u$ implies that
$\| u \ph (1 \otimes a) u^* - a \| < \ep_0 \leq \frac{\ep}{2}.$
Therefore $\| a q_l - q_l a \| < \ep.$
The claim is proved.

It remain to prove that
$\| \bt_{\dt_k} (a) - \af_{\dt_k} (a) \| < \dt$
for $k = 1, 2, \ldots, d$ and for all $a \in T.$
For any $l \in \Z^d$ and $a \in T,$
we have
\[
\| \bt_{l} (a) - \af_{l} (a) \|
  = \big\| u \ph \big( (\gm_{l} \otimes \af_{l})
                     \big( \ph^{-1} (u^* a u) \big) \big)
         - \af_{l} (a) \big\|.
\]
Now
\[
\| \ph^{-1} (u^* a u) - 1 \otimes a \|
  = \| a - u \ph (1 \otimes a) u^* \|
  < \ep_0
\]
and, using $\dt_k (a) \in S,$
\[
\big\| u \ph \big( (\gm_{\dt_k} \otimes \af_{\dt_k})
                 ( 1 \otimes a ) \big) u^*
         - \af_{\dt_k} (a) \big\|
 = \big\| u \ph ( 1 \otimes \af_{\dt_k} (a) ) u^*
                             - \af_{\dt_k} (a) \big\|
 < \ep_0.
\]
We conclude that
$\| \bt_{\dt_k} (a) - \af_{\dt_k} (a) \| < \ep_0 + \ep_0 \leq \dt,$
as desired.
This completes the proof.
\end{proof}

\begin{thm}\label{T-MainZd}
Let $A$ be a separable unital \ca\  %
such that ${\mathcal{O}}_{\I} \otimes A \cong A,$
and let $d \in \N.$
Then there exists a dense $G_{\dt}$-set in ${\mathrm{Act}}_d (A)$
consisting of actions with the Rokhlin property.
\end{thm}

\begin{proof}
This follows immediately from
Lemmas \ref{L-2910PIOpen}, \ref{L-2910PIIntersect},
and~\ref{L-2910PIApprox},
using the Baire Category Theorem
and the fact that the $2^d$-tuple $m$ of Lemma~\ref{L-2910PIApprox}
has arbitrarily large minimum height.
\end{proof}

\begin{cor}\label{T-Main3}
Let $A$ be a separable unital \ca\  %
such that ${\mathcal{O}}_{\I} \otimes A \cong A.$
Then there exists a dense $G_{\dt}$-set in $\Aut (A)$
consisting of automorphisms with the Rokhlin property
as in Definition~\ref{D-IzumiRP}.
\end{cor}

\begin{proof}
This follows immediately from
Theorem~\ref{T-MainZd}
and Lemma~\ref{R-2909RZdRZ}.
\end{proof}

\begin{thm}\label{T-RPUHF}
Let $r \in \N,$
and let $D = \bigotimes_{n = 1}^{\I} M_r$
be the $r^{\infty}$~UHF algebra.
Let $A$ be a separable unital \ca\  %
such that $D \otimes A \cong A,$
and let $d \in \N.$
Then there exists a dense $G_{\dt}$-set in ${\mathrm{Act}}_d (A)$
consisting of actions with the Rokhlin property.
\end{thm}

\begin{proof}
The proof is essentially the same as that of Theorem~\ref{T-MainZd}.
The difference is in Lemma~\ref{L-2910PIApprox},
where we take $m$ to be the $1$-tuple consisting of
$(r^n, r^n, \ldots, r^n) \in \Z^d.$
We still have an isomorphism $\ph \colon D^{\otimes d} \otimes A \to A$
which is approximately unitarily equivalent to $a \mapsto 1 \otimes a.$
We let $v \in M_{r^n}$ be a cyclic permutation matrix
and let the Rokhlin tower consist of the rank one \pj s which are
permuted.
We take
\[
w_k = 1 \otimes 1 \otimes \cdots 1 \otimes v \otimes 1 \otimes 1
            \otimes \cdots  \otimes 1
\]
with $v$ in the $k$th position.
We take $\ph_0 \colon M_r \to D$ to be the inclusion
$x \mapsto x \otimes 1$ of the first tensor factor.
Then $\gm_{\dt_k} \in \Aut (D)$ is $\Ad (\rh_1 (w_k)).$
\end{proof}

\begin{cor}\label{C-ZRPUHF}
Let $A$ be as in Theorem~\ref{T-RPUHF}.
Then there exists a dense $G_{\dt}$-set in ${\mathrm{Act}} (A)$
consisting of automorphisms with the Rokhlin property
as in Definition~\ref{D-IzumiRP}.
\end{cor}

\begin{proof}
This follows immediately from
Theorem~\ref{T-RPUHF}
and Lemma~\ref{R-2909RZdRZ}.
\end{proof}

\begin{thm}\label{T-RPFinGp}
Let $G$ be a finite group of cardinality~$r,$
and let $D$ be the $r^{\infty}$~UHF algebra.
Let $A$ be a separable unital \ca,
and let ${\mathrm{Act}}_G (A)$ denote the set of actions of~$G$
on~$A.$
For any enumeration $S = (a_1, a_2, \ldots )$ of a countable dense
subset of~$A,$
the formula
\[
\rh_{G, S} (\af, \bt) = \sup_{g \in G} \rh (\af_g, \bt_g)
\]
defines a complete metric on ${\mathrm{Act}}_G (A)$
which induces the restriction to ${\mathrm{Act}}_G (A)$
of the product topology on $\Aut (A)^G.$
If $A$ is unital and $D \otimes A \cong A,$
then the set of actions with the Rokhlin property
is a dense $G_{\dt}$-set in ${\mathrm{Act}}_G (A).$
\end{thm}

\begin{proof}
The proof of the part about the metric is the same as the proof
of Lemma~\ref{L-2910PIisMet}.

The proof of the second part
is similar to that of Theorem~\ref{T-MainZd}
and the lemmas used there.
For a finite set $F \S A$ and $\ep > 0,$
we define $W_G (F, \ep) \S {\mathrm{Act}}_G (A)$
to be the set of all
$\af \in {\mathrm{Act}}_G (A)$ such that
there are \pj s $e_g \in A$ for $g \in G$
with $\sum_{g \in G} e_g = 1,$
with $\| \af_g (e_h) - e_{g h} \| < \ep$ for all $g, h \in G,$
and with $\| e_g a - a e_g \| < \ep$ for all $g \in G$ and $a \in A.$
The sets $W_G (F, \ep)$ are open
by an argument similar to the proof of Lemma~\ref{Open}.
Choose a countable dense subset $S \S A,$
and let ${\mathcal{F}}$ be the set of all finite subsets of~$S.$
One checks that
$\af \in {\mathrm{Act}}_G (A)$
has the Rokhlin property \ifo\  %
$\af \in \bigcap_{F \in {\mathcal{F}}}
            \bigcap_{r = 1}^{\I} W_G \left( F, \tfrac{1}{r} \right).$

In the proof of the analog of Lemma~\ref{L-2910PIApprox},
we use an isomorphism $\ph \colon D \otimes A \to A$
(rather than from $D^{\otimes d} \otimes A$)
which is approximately unitarily equivalent to $a \mapsto 1 \otimes a.$
We take $\ph_1 \colon M_r \to D$ to be the inclusion
$x \mapsto x \otimes 1$ of the first tensor factor.
We define an action $\gm \colon G \to \Aut (D)$ as follows.
Identify $M_r$ with $L (l^2 (G)),$
let $g \mapsto w_g \in M_r$ be the left regular representation of~$G,$
and take $\gm_g = \Ad (\ph_1 (w_g)).$
The Rokhlin tower in $M_r$
is taken to be $(e_g)_{g \in G}$ with $e_g$ being the projection
onto the span of the standard basis vector in $l^2 (G)$
corresponding to $g \in G.$
The rest of the proof of Lemma~\ref{L-2910PIApprox}
goes through with the obvious changes.
\end{proof}

\end{document}